\documentclass[11pt]{amsart}
\usepackage[headings]{fullpage}
\usepackage{amssymb,amsmath,mathtools,bbm,tikz,tikz-cd,color,relsize}
\usepackage[all,cmtip]{xy}
\usepackage{url}
\usetikzlibrary{positioning}
\usetikzlibrary{calc}
\usetikzlibrary{decorations.pathreplacing}
\usetikzlibrary{arrows}
\usetikzlibrary{calc}
\usetikzlibrary{decorations.markings}
\usetikzlibrary{ext.arrows-plus}
\numberwithin{equation}{section}

\tikzstyle{hvector}=[inner sep=2pt,draw=blue!50,fill=blue!10,thick]
\tikzstyle{unit}=[inner sep=2pt,shape=circle, draw]
\tikzstyle{counit}=[inner sep=2pt,shape=circle, draw,fill=gray]
\tikzstyle{antipode}=[inner sep=2pt,shape=rectangle, draw]
\tikzstyle{fantipode}=[inner sep=2pt,shape=rectangle, draw,fill,white]
\tikzstyle{cocycle}=[inner sep=2pt,shape=circle, draw]
\tikzstyle{twistedm}=[inner sep=2pt,shape=circle, fill=gray]
\tikzstyle{autom}=[inner sep=2pt,shape=circle, draw]
\tikzstyle{coact}=[inner sep=2pt,shape=circle, fill=black]
\tikzstyle{wrec}=[inner sep=2pt,shape=rectangle, minimum width=6pt, minimum height=4pt]
\newcommand{\bell}[3]
{(#1) ellipse (#2 and #3)}

\tikzset{
  on each segment/.style={
    decorate,
    decoration={
      show path construction,
      moveto code={},
      lineto code={
        \path [#1]
        (\tikzinputsegmentfirst) -- (\tikzinputsegmentlast);
      },
      curveto code={
        \path [#1] (\tikzinputsegmentfirst)
        .. controls
        (\tikzinputsegmentsupporta) and (\tikzinputsegmentsupportb)
        ..
        (\tikzinputsegmentlast);
      },
      closepath code={
        \path [#1]
        (\tikzinputsegmentfirst) -- (\tikzinputsegmentlast);
      },
    },
    },
 mid arrow/.style={
    postaction={decorate},
    decoration={
      markings,
      mark=at position 0.5 with {
        \arrow*[xshift=0pt,  thick,scale=1]{stealth}
      }}},
       end arrow/.style={
    postaction={decorate},
    decoration={
      markings,
      mark=at position 0.95 with {
        \arrow*[xshift=0pt,  thick,scale=1]{stealth}
      }}},
  mid end arrow/.style={
    postaction={decorate},
    decoration={
      markings,
      mark=at position 0.7 with {
        \arrow*[xshift=0pt,  thick,scale=1]{stealth}
      }}},
  start arrow/.style={
    postaction={decorate},
    decoration={
      markings,
      mark=at position 0.3 with {
        \arrow*[xshift=0pt,  thick,scale=1]{stealth}
      }}}
      }

\usepackage[bookmarks=true,%
    colorlinks=true,%
    linkcolor=blue,%
    citecolor=blue,%
    filecolor=blue,%
    menucolor=blue,%
    urlcolor=blue,%
    breaklinks=true]{hyperref}
\usepackage{slashed}
\usepackage{verbatim}
\usepackage[normalem]{ulem}  %
\usepackage{diagbox}

\usepackage{amssymb,amsmath,tikz}
\newcommand{\B}{\mathbb{B}}
\newcommand{\Z}{\mathbb{Z}}

\newcommand{\lcf}{\mathbb{F}}

\newcommand{\qbinom}[3]{\genfrac{[}{]}{0pt}{}{#1}{#2}_{#3}}
\newcommand{\gauss}[1]{\gamma_{#1}}
\newtheorem{theorem}{Theorem}
\newtheorem{proposition}{Proposition}
\newtheorem{lemma}{Lemma}

\theoremstyle{remark}

\newtheorem{remark}{Remark}
\newtheorem{conjecture}{Conjecture}

\def\be{\begin{equation}}
\def\ee{\end{equation}}
\newcommand{\antipode}{S}
\newcommand{\fP}{\mathcal{P}_{\mathrm{fin}}}

\title{On braided Hopf structures on exterior algebras}
\author{Rinat Kashaev}
\address{Section de math\'ematiques, Universit\'e de Gen\`eve,
rue du Conseil-G\'en\'eral 7-9, 1205 Gen\`eve, Suisse\\}
\email{Rinat.Kashaev@unige.ch}
\author{Vladimir Mangazeev}
\address{Department of Fundamental \& Theoretical Physics, Research School of Physics,
 Australian National University, Canberra, ACT 2601, Australia\\}
\email{Vladimir.Mangazeev@anu.edu.au}
\date{June 20, 2025}

\begin{document}

\begin{abstract}
We show that the exterior algebra of a vector space
$V$ of dimension greater than one admits a one-parameter family of braided Hopf algebra structures, arising from its identification with a Nichols algebra. We explicitly compute the structure constants with respect to a natural set-theoretic basis.

A one-parameter family of diagonal automorphisms exists, which we use to construct solutions to the (constant) Yang--Baxter equation. These solutions are conjectured to give rise to  the two-variable Links–Gould polynomial invariants associated with the super-quantum group
 \( U_q(\mathfrak{gl}(N|1)) \), where \( N = \dim(V) \). We support this conjecture through computations for small values of
$N$.
\end{abstract}
\maketitle

\section{Introduction}
Let $(V,\tau)$ be a braided $\lcf$-vector space. This means that $\tau\in\operatorname{Aut}(V\otimes V)$ satisfies the Yang--Baxter equation
\begin{equation}\label{eq:ybe}
\tau'\tau''\tau'=\tau''\tau'\tau'',\quad \tau':=\tau\otimes\operatorname{id}_V,\  \tau'':=\operatorname{id}_V\otimes\tau.
\end{equation}
In this case, the tensor algebra $T(V)=\bigoplus_{n=0}^\infty V^{\otimes n}$ is canonically a braided Hopf algebra with an invertible antipode and where all the elements of $V$ are primitive. The Nichols algebra associated with $(V,\tau)$ is the quotient braided Hopf algebra $\mathfrak{B}(V)=T(V)/\mathfrak{J}_V$ where $\mathfrak{J}_V$ is the maximal Hopf ideal intersecting trivially with $V$, see, for example \cite{MR4164719}. It is a two-sided ideal generated by primitive elements of degree two and higher, where elements of degree $n$ are those of $T^n(V):=V^{\otimes n}$.

A well known example of Nichols algebra is the exterior algebra $\bigwedge \!\!V$ corresponding to the diagonal braiding
\be\label{superbr}
\tau(u\otimes v)=-v\otimes u,\quad \forall u,v\in V.
\ee
The algebra structure is given by the presentation
\be
\mathfrak{B}(V)\simeq\bigwedge \!\!V=\lcf\langle V\mid x^2=0,\ \forall x\in V\rangle,
\ee
see~\cite[Example~1.10.15]{MR4164719}. A nice property of the exterior algebra is that its dimension is $2^{\operatorname{dim}(V)}$ which is finite when the dimension of $V$ is finite.

In this work, we show that the exterior algebra of any vector space $V$ with $\dim(V)\ge2$ admits a one-parameter family of deformed braided Hopf algebra structures, denoted $\Lambda_p(V)$, whose braiding is induced by suitably normalised Hecke-type quantum $R$-matrices. The underlying algebra of $\Lambda_p(V)$ remains undeformed. Motivated by Kuperberg's spiders~\cite {MR1403861}, the Murakami--Ohtsuki--Yamada (MOY) diagrammatic calculus~\cite{MR1659228},  and the work  of Cautis--Kamnitzer--Morrison~\cite{MR3263166}, we analyse the $U_q(\mathfrak{gl}_2)$-module structure on $\Lambda_p(V)^{\otimes 2}$ (with $p=q^2$) and derive a streamlined formula for the braiding on $\Lambda_p(V)$, which corresponds to the formulas in~\cite[Theorem 5.1]{MR1659228} and in~\cite[Corollary 6.2.3]{MR3263166}, and calculate explicitly the matrix coefficients in the  set-theoretic basis.

Following the approach of~\cite{GaroufalidisKashaev2023}, we construct  $R$-matrices associated with diagonal automorphisms and the corresponding right generalised Yetter--Drinfel'd module structures on $\Lambda_p(V)$. Using the MOY calculus, we obtain a compact expression for these $R$-matrices and an explicit formula for their matrix coefficients in the set-theoretic basis. For an $N$-dimensional input vector space, the knot invariant produced by the resulting $R$-matrix is conjectured to coincide with the  two-variable Links--Gould polynomial arising from the super quantum group $U_q(\mathfrak{gl}(N|1))$~\cite{MR1199742,MR1223526}.

 \subsection*{Outline} In Section~\ref{sec:0}, we recall the definition of braided Hopf objects and braided Hopf algebras, and introduce the string-diagrammatic notation.
 In Section~\ref{sec:1}, we prove Theorems~\ref{thm:TV/J2}, \ref{thm:1}, and \ref{thm:Lambda-p-Nichols}.
  In Section~\ref{sec:5}, using the MOY calculus, we derive a simplified formula for the braiding of $\Lambda_p(V)$. In Section~\ref{sec:2}, we further analyse the structure of the braiding of  $\Lambda_p(V)$, and we obtain an explicit expression for its matrix coefficients.
 In Section~\ref{sec:3}, after introducing diagonal automorphisms of $\Lambda_p(V)$, we apply the MOY calculus to prove Theorem~\ref{thm:big-R-matrix}. We then work out the explicit form of the resulting $R$-matrices for a few small values of $\dim(V)$.

 \subsection*{Acknowledgements}
 The authors would like to thank Vladimir Bazhanov, Jehanne Dousse, Stavros Garoufalidis, Thomas Kerler, Thang L\^e, Gregor Masbaum, Roland van der Veen, Alexis Virelizier,  and  Emmanuel Wagner for fruitful discussions.
Some of these results were presented at the ``Hot Topics'' conference on Quantum Topology in Bonn, May 12--16, 2025. The authors thank the participants for stimulating discussions, especially Francis~Bonahon and L\`eo~Schelstraete for pointing out previous work on related topics. Special thanks go to Nicol\'as Andruskiewitsch and Jo\~ao Matheus~Jury Giraldi for bringing their work~\cite{MR3775317} to our attention, which overlaps with Theorems~\ref{thm:TV/J2} and \ref{thm:Lambda-p-Nichols} when $\dim(V)=2$.

VM gratefully acknowledges the hospitality of the University of Geneva during February--March 2025, when part of this work was completed.

This work is partially supported by the SNSF research program NCCR The Mathematics of Physics (SwissMAP), the ERC Synergy Grant Recursive and Exact New Quantum Theory (ReNew Quantum), and the SNSF grants no. 200021-232258 and no. 200020-200400.

\section{Braided Hopf algebras and a diagrammatic language}\label{sec:0}

In this section, we briefly review the definition of a braided Hopf algebra, first introduced by Shahn Majid in~\cite{MR1289422}. We begin with the abstract notion of a braided Hopf object in a braided monoidal category, and then specialise to the context of vector spaces.

Let $\mathcal{C}=(\mathcal{C},\otimes, I, \beta)$ be a (strict) braided monoidal category with tensor product $\otimes$, unit object $I$, and braiding  $\beta:\otimes\to\otimes^{\mbox{\tiny op}}$, see \cite{MR3674995}  for definitions.

A \emph{\color{blue} (braided) Hopf object} in $\mathcal{C}$  is an object $H$ endowed with five structural morphisms:
\be
\label{eq:hopfmaps}
\nabla\colon H\otimes H\to H,\quad \eta\colon I\to H,\quad
\Delta\colon H\to H\otimes H,\quad \epsilon\colon H\to I,
\quad S\colon H\to H
\ee
called, respectively, the product, unit, coproduct, counit and antipode, that satisfy the following relations or axioms
\begin{subequations}
\begin{align}
\label{Hr1}
\text{associativity :} & \hspace{0.2cm}
\nabla(\nabla \otimes \operatorname{id}) = \nabla(\operatorname{id} \otimes \nabla)
\\
\label{Hr2}
\text{coassociativity :} & \hspace{0.2cm}
(\Delta \otimes \operatorname{id})\Delta = (\operatorname{id} \otimes \Delta)\Delta
\\
\label{Hr3}
\text{unitality :} & \hspace{0.2cm}
\nabla(\eta \otimes \operatorname{id}) = \operatorname{id} =
\nabla(\operatorname{id} \otimes \eta)
\\
\label{Hr4}
\text{counitality :} & \hspace{0.2cm}
(\epsilon \otimes \operatorname{id})\Delta = \operatorname{id} =
(\operatorname{id} \otimes \epsilon)\Delta
\\
\label{Hr5}
\text{invertibility :} & \hspace{0.2cm}
\nabla(\operatorname{id} \otimes S) \Delta = \eta \epsilon =
\nabla(S \otimes \operatorname{id}) \Delta\\
\label{Hr7}
\text{compatibility :} & \hspace{0.2cm}
(\nabla \otimes \nabla)(\operatorname{id} \otimes \beta_{H,H} \otimes
\operatorname{id})(\Delta \otimes \Delta) = \Delta \nabla
\end{align}
\end{subequations}
where we omit the composition symbol.

In our case, the braided monoidal category
$\mathcal{C}$ is a subcategory of the category
$\mathbf{Vect}_\lcf$ of vector spaces over a field $\lcf$, equipped with the monoidal
structure given by the tensor product $\otimes_\lcf$ and the unit object $I=\lcf$. In this setting, the objects and morphisms of $\mathcal{C}$
are  called \emph{\color{blue} braided vector spaces} and  \emph{\color{blue} (braided linear) maps}, respectively,
whereas a Hopf object is referred to as
a \emph{\color{blue} braided Hopf algebra}, see the review paper~\cite{Takeuchi}.

Following \cite[Chapter 2]{MR3674995}, we introduce the following graphical notation of string diagrams for the structural morphisms
of a braided Hopf algebra $H$ (all lines correspond to $H$)
\begin{equation}
\nabla\ =\
\begin{tikzpicture}[scale=1.6,baseline=-2]
\draw[thick]  (0,0)--(0,10pt)(0,0) to [out=-135,in=90] (-5pt,-10pt) (0,0)
  to [out=-45,in=90] (5pt,-10pt);
\end{tikzpicture}
\quad (\text{product}),\qquad
\Delta\ =\
\begin{tikzpicture}[scale=1.65,baseline=-2]
\draw[thick]  (0,0)--(0,-10pt) (0,0) to [out=135,in=-90] (-5pt,10pt) (0,0)
  to [out=45,in=-90] (5pt,10pt);
\end{tikzpicture}
\quad (\text{coproduct}),\qquad
\beta_{H,H}=
\begin{tikzpicture}[baseline=13,yscale=1.1]
\draw[thick] (0,1) to [out=-90,in=90] (1,0);
\draw[line width=3pt,white] (1,1) to [out=-90,in=90] (0,0);
\draw[thick] (1,1) to [out=-90,in=90] (0,0);
\end{tikzpicture}\quad (\text{braiding})
\end{equation}
\begin{equation}
\eta\ =\
\begin{tikzpicture}[scale=1.5,baseline=2]
\node (x) [unit]{};
\draw[thick]  (x)--+(0,10pt);
\end{tikzpicture}
\quad (\text{unit}),\qquad
\epsilon\ =\
\begin{tikzpicture}[scale=1.5,baseline=-7]
\node (x)[counit] {};
\draw[thick]  (x)--+(0,-10pt);
\end{tikzpicture}
\quad (\text{counit}),\qquad
\antipode\ =\
\begin{tikzpicture}[scale=1.5,baseline=-3]
\node (x)[antipode]{};
\draw[thick]  (x)--+(0,10pt)(x)--+(0,-10pt);
\end{tikzpicture}
\quad (\text{antipode}).
\end{equation}

The relations or axioms of a braided Hopf algebra take the following graphical form:
\begin{equation}
\begin{tikzpicture}[scale=2,baseline=10pt]
\draw[thick] (0,0) to [out=90,in=-135](5pt,5pt)(10pt,0) to [out=90,in=-45](5pt,5pt)
(5pt,5pt) to [out=90,in=-135](10pt,10pt)(20pt,0) to [out=90,in=-45](10pt,10pt)
(10pt,10pt)--(10pt,15pt);
\end{tikzpicture}
\ =\
\begin{tikzpicture}[scale=2,baseline=10pt]
\draw[thick] (0,0) to [out=90,in=-135] (10pt,10pt) (10pt,0)
to [out=90,in=-135](15pt,5pt)
(15pt,5pt) to [out=90,in=-45] (10pt,10pt) (20pt,0) to [out=90,in=-45] (15pt,5pt)
(10pt,10pt)--(10pt,15pt);
\end{tikzpicture} \quad(\text{associativity}),
\qquad
\begin{tikzpicture}[scale=2,baseline=6pt]
\node (x) [unit]{};
\draw[thick] (x) to [out=90,in=-135] +(5pt,5pt);
\draw[thick] (x)+(10pt,0) to [out=90,in=-45] ++(5pt,5pt)--+(0,5pt);
\end{tikzpicture}
\ =\
\begin{tikzpicture}[scale=2,baseline=6pt]
\draw[thick] (0,0)--(0,10pt);
\end{tikzpicture}
\ =\
\begin{tikzpicture}[scale=2,baseline=6pt]
\node (x) [unit]{};
\draw[thick] (x) to [out=90,in=-45] +(-5pt,5pt);
\draw[thick] (x)+(-10pt,0) to [out=90,in=-135] ++(-5pt,5pt)--+(0,5pt);
\end{tikzpicture}
\quad (\text{unitality}),
\end{equation}

\begin{equation}
\begin{tikzpicture}[scale=2,baseline=10pt]
\draw[thick] (0,0)--(0,5pt) to [out=135,in=-90] (-5pt,10pt)
to [out=135,in=-90] (-10pt,15pt)
(-5pt,10pt) to [out=45,in=-90] (0,15pt) (0,5pt) to [out=45,in=-90] (10pt,15pt);
\end{tikzpicture}
\ =\
\begin{tikzpicture}[scale=2,baseline=10pt]
\draw[thick] (0,0)--(0,5pt) to [out=45,in=-90] (5pt,10pt) to [out=45,in=-90] (10pt,15pt)
(5pt,10pt) to [out=135,in=-90] (0,15pt) (0,5pt) to [out=135,in=-90] (-10pt,15pt);
\end{tikzpicture}
\quad (\text{coassociativity}),
\qquad
\begin{tikzpicture}[scale=2,baseline=-12]
\node (x) [counit]{};
\draw[thick] (x) to [out=-90,in=135] +(5pt,-5pt);
\draw[thick] (x)+(10pt,0) to [out=-90,in=45] ++(5pt,-5pt)--+(0,-5pt);
\end{tikzpicture}
\ =\
\begin{tikzpicture}[scale=2,baseline=8]
\draw[thick] (0,0)--(0,10pt);
\end{tikzpicture}
\ =\
\begin{tikzpicture}[scale=2,baseline=-12]
\node (x) [counit]{};
\draw[thick] (x) to [out=-90,in=45] +(-5pt,-5pt);
\draw[thick] (x)+(-10pt,0) to [out=-90,in=135] ++(-5pt,-5pt)--+(0,-5pt);
\end{tikzpicture}\quad
(\text{counitality}),
\end{equation}

\begin{equation}
\begin{tikzpicture}[scale=2,baseline=-3]
\node (x)[antipode]{};
\draw[thick] (x) to [out=90,in=-135](3pt,5pt)(x) to [out=-90,in=135] (3pt,-5pt)
-- (3pt,-10pt) -- (3pt,-5pt) .. controls +(4.5pt,2pt) and +(4.5pt,-2pt) ..  (3pt,5pt) -- (3pt,10pt)
;
\end{tikzpicture}
\ =\
\begin{tikzpicture}[scale=2,baseline=-10]
\node (x)[unit]{};
\path (x)+(0,-7pt) node (y) [counit]{};
\draw[thick] (x)--+(0,7pt);
\draw[thick] (y)--+(0,-7pt);
\end{tikzpicture}
\ =\
\begin{tikzpicture}[scale=2,baseline=-3]
\node (x)[antipode]{};
\draw[thick] (x) to [out=90,in=-45](-3pt,5pt) (x) to [out=-90,in=45] (-3pt,-5pt)
-- (-3pt,-10pt) -- (-3pt,-5pt) .. controls +(-4.5pt,2pt) and +(-4.5pt,-2pt) ..  (-3pt,5pt) -- (-3pt,10pt);
\end{tikzpicture}\quad
(\text{invertibility}),
\end{equation}
and
\begin{equation}\label{graph-comp1}
\begin{tikzpicture}[scale=3,baseline=20]
\draw[thick] (0,0)--(0,5pt) to [out=135,in=-135] (0,10pt)--(0,15pt)
(10pt,0)--(10pt,5pt) to [out=45,in=-45] (10pt,10pt)--(10pt,15pt)
(0,5pt)--(10pt,10pt) (10pt,5pt)--(5.5pt,7.25pt)(4.5pt,7.75pt)--(0pt,10pt);
\end{tikzpicture}
\quad\ =\
\begin{tikzpicture}[scale=3,baseline=20]
\draw[thick] (0,0) to [out=90,in=-135] (5pt,5pt)--(5pt,10pt)
to [out=135,in=-90] (0,15pt)
(10pt,0) to [out=90,in=-45] (5pt,5pt) (5pt,10pt) to [out=45,in=-90] (10pt,15pt);
\end{tikzpicture}\quad
(\text{compatibility}).
\end{equation}

The fact that all morphisms are braided linear maps (including the braiding $\beta_{H,H}$ itself) means that the following graphical relations hold:
\begin{equation}\label{graph-comp}
\begin{tikzpicture}[baseline=13,yscale=1.1]
\draw[thick] (0,1) to [out=-90,in=90] (1,0);
\draw[line width=3pt,white] (1,1.05) to [out=-90,in=90] (0,0.05);
 \node (x) [unit] at (0,0)  {}  ;
\draw[thick] (x.north) to [out=90, in=-90] (1,1);
\end{tikzpicture}
\ =\
\begin{tikzpicture}[baseline=13,yscale=1.1]
\draw[thick] (0,1) to [out=-90,in=90] (1,0);
 \node (x) [unit] at (0.65,0.65)  {}  ;
\draw[thick] (x.north east) to [out=45, in=-90] (1,1);
\end{tikzpicture}
,\qquad
\begin{tikzpicture}[baseline=13,yscale=1.1]
 \node (x) [unit] at (1,0)  {}  ;
\draw[thick] (x.north) to [out=90, in=-90] (0,1);
\draw[line width=3pt,white] (1,1) to [out=-90,in=90] (0,0);
\draw[thick] (0,0) to [out=90,in=-90] (1,1);
\end{tikzpicture}
\ =\
\begin{tikzpicture}[baseline=13,yscale=1.1]
 \node (x) [unit] at (0.35,0.65)  {}  ;
\draw[thick] (x.north west) to [out=135, in=-90] (0,1);
\draw[line width=3pt,white] (1,1) to [out=-90,in=90] (0,0);
\draw[thick] (0,0) to [out=90,in=-90] (1,1);
\end{tikzpicture},
\end{equation}
\begin{equation}\label{graph-co-comp}
\begin{tikzpicture}[baseline=-18,yscale=-1.1,xscale=-1]
\draw[thick] (0,1) to [out=-90,in=90] (1,0);
\draw[line width=3pt,white] (1,1.05) to [out=-90,in=90] (0,0.05);
 \node (x) [counit] at (0,0)  {}  ;
\draw[thick] (x.south) to [out=90, in=-90] (1,1);
\end{tikzpicture}
\ =\
\begin{tikzpicture}[baseline=-18,yscale=-1.1,xscale=-1]
\draw[thick] (0,1) to [out=-90,in=90] (1,0);
 \node (x) [counit] at (0.65,0.65)  {}  ;
\draw[thick] (x.south west) to [out=45, in=-90] (1,1);
\end{tikzpicture}
,\qquad
\begin{tikzpicture}[baseline=-18,yscale=-1.1,xscale=-1]
 \node (x) [counit] at (1,0)  {}  ;
\draw[thick] (x.south) to [out=90, in=-90] (0,1);
\draw[line width=3pt,white] (1,1) to [out=-90,in=90] (0,0);
\draw[thick] (0,0) to [out=90,in=-90] (1,1);
\end{tikzpicture}
\ =\
\begin{tikzpicture}[baseline=-18,yscale=-1.1,xscale=-1]
 \node (x) [counit] at (0.35,0.65)  {}  ;
\draw[thick] (x.south east) to [out=135, in=-90] (0,1);
\draw[line width=3pt,white] (1,1) to [out=-90,in=90] (0,0);
\draw[thick] (0,0) to [out=90,in=-90] (1,1);
\end{tikzpicture},
\end{equation}
\begin{equation}\label{graph-ant-comp}
\begin{tikzpicture}[baseline=13,yscale=1.1]
\draw[thick] (0,1) to [out=-90,in=90] (1,0);
\node (x) [antipode] at (1/8,1/4)  {}  ;
\draw[line width=3pt,white] (x.north) to [out=90,in=-90] (1,1);
 \draw[thick] (x.north) to [out=90, in=-90] (1,1);
\draw[thick] (x.south) to [out=-90, in=90] (0,0);
\end{tikzpicture}
\ =\
\begin{tikzpicture}[baseline=-18,yscale=-1.1,xscale=-1]
\draw[thick] (0,1) to [out=-90,in=90] (1,0);
\node (x) [antipode] at (1/8,1/4)  {}  ;
\draw[line width=3pt,white] (x.south) to [out=90,in=-90] (1,1);
 \draw[thick] (x.south) to [out=90, in=-90] (1,1);
\draw[thick] (x.north) to [out=-90, in=90] (0,0);
\end{tikzpicture}
,\qquad
\begin{tikzpicture}[baseline=13,yscale=1.1]
 \node (x) [antipode] at (7/8,1/4)  {}  ;
\draw[thick] (x.north) to [out=90, in=-90] (0,1);
\draw[thick] (x.south) to [out=-90, in=90] (1,0);
\draw[line width=3pt,white] (1,1) to [out=-90,in=90] (0,0);
\draw[thick] (0,0) to [out=90,in=-90] (1,1);
\end{tikzpicture}
\ =\
\begin{tikzpicture}[baseline=-18,yscale=-1.1,xscale=-1]
 \node (x) [antipode] at (7/8,1/4)  {}  ;
\draw[thick] (x.south) to [out=90, in=-90] (0,1);
\draw[thick] (x.north) to [out=-90, in=90] (1,0);
\draw[line width=3pt,white] (1,1) to [out=-90,in=90] (0,0);
\draw[thick] (0,0) to [out=90,in=-90] (1,1);
\end{tikzpicture},
\end{equation}
\begin{equation}\label{eq:braiding-product}
\begin{tikzpicture}[scale=1.5,baseline=20,x=1pt,y=1pt]
\draw[thick]  (0,13)--(0,17)(0,13) to [out=-135,in=90] (-5,3) (0,13)
  to [out=-45,in=90] (5,3);
\draw[thick] (15,3) --  (15,12) to  [out=90,in = -90] (-5,25);
\draw[line width=2.5pt,white] (0,15) to [out=90,in=-90] (10,25);
\draw[thick] (0,15) to [out=90,in=-90] (10,25);
    \end{tikzpicture}
\ =\
\begin{tikzpicture}[scale=1.5,baseline=20,x=1pt,y=1pt]
\draw[thick] (15,3) to [out=90,in=-90] (-5,25);
\draw[line width=2.5pt,white](10,21) to [out=-135,in=90] (-5pt,3pt);
\draw[line width=2.5pt,white] (5,3) to [out=90,in=-90] (15,17);
\draw[thick] (5,3) to [out=90,in=-90] (15,17);
\draw[thick]  (10,21)--(10,25)(10,21) to [out=-135,in=90] (-5pt,3pt)
(10,21) to [out=-30,in=90] (15,17);
    \end{tikzpicture},
\qquad
\begin{tikzpicture}[scale=1.5,baseline=20,x=1pt,y=1pt]
\draw[thick] (10,15) to [out=90,in=-90] (0,25);
\draw[thick]  (10,13)--(10,15)(10,13) to [out=-45,in=90] (15,3) (10,13)
  to [out=-135,in=90] (5,3);
\draw[line width=2.5pt,white] (-5,12) to [out=90,in = -90] (15,25);
\draw[thick] (-5,3) --  (-5,12) to  [out=90,in = -90] (15,25);
    \end{tikzpicture}
\ =\
\begin{tikzpicture}[scale=1.5,baseline=20,x=1pt,y=1pt]
\draw[thick] (5,3) to [out=90,in=-90] (-5,17);
\draw[thick]  (0,21)--(0,25)(0,21) to [out=-45,in=90] (15,3)
(0,21) to [out=-120,in=90] (-5,17);
\draw[line width=2.5pt,white] (-5,3) to [out=90,in=-90] (15,25);
\draw[thick] (-5,3) to [out=90,in=-90] (15,25);
    \end{tikzpicture},
\end{equation}
\begin{equation}
\begin{tikzpicture}[scale=-1.5,baseline=-20,x=1pt,y=1pt]
\draw[thick]  (0,13)--(0,17)(0,13) to [out=-135,in=90] (-5,3) (0,13)
  to [out=-45,in=90] (5,3);
\draw[thick] (15,3) --  (15,12) to  [out=90,in = -90] (-5,25);
\draw[line width=2.5pt,white] (0,15) to [out=90,in=-90] (10,25);
\draw[thick] (0,15) to [out=90,in=-90] (10,25);
    \end{tikzpicture}
\ =\
\begin{tikzpicture}[scale=-1.5,baseline=-20,x=1pt,y=1pt]
\draw[thick] (15,3) to [out=90,in=-90] (-5,25);
\draw[line width=2.5pt,white](10,21) to [out=-135,in=90] (-5pt,3pt);
\draw[line width=2.5pt,white] (5,3) to [out=90,in=-90] (15,17);
\draw[thick] (5,3) to [out=90,in=-90] (15,17);
\draw[thick]  (10,21)--(10,25)(10,21) to [out=-135,in=90] (-5pt,3pt)
(10,21) to [out=-30,in=90] (15,17);
    \end{tikzpicture},
\qquad
\begin{tikzpicture}[scale=-1.5,baseline=-20,x=1pt,y=1pt]
\draw[thick] (10,15) to [out=90,in=-90] (0,25);
\draw[thick]  (10,13)--(10,15)(10,13) to [out=-45,in=90] (15,3) (10,13)
  to [out=-135,in=90] (5,3);
\draw[line width=2.5pt,white] (-5,12) to [out=90,in = -90] (15,25);
\draw[thick] (-5,3) --  (-5,12) to  [out=90,in = -90] (15,25);
    \end{tikzpicture}
\ =\
\begin{tikzpicture}[scale=-1.5,baseline=-20,x=1pt,y=1pt]
\draw[thick] (5,3) to [out=90,in=-90] (-5,17);
\draw[thick]  (0,21)--(0,25)(0,21) to [out=-45,in=90] (15,3)
(0,21) to [out=-120,in=90] (-5,17);
\draw[line width=2.5pt,white] (-5,3) to [out=90,in=-90] (15,25);
\draw[thick] (-5,3) to [out=90,in=-90] (15,25);
    \end{tikzpicture},
\end{equation}
and the Yang--Baxter equation for the braiding
\begin{equation}
\begin{tikzpicture}[baseline=25]
\coordinate (m1) at (0.5,1);
\draw[thick] (2,0) to  [out=90,in = -90] (0,2);
\draw[line width=2.5pt,white] (1,0) to[out=90,in=-90] (m1) to[out=90,in=-90] (1,2);
\draw[thick] (1,0) to [out=90,in=-90] (m1) to [out=90,in=-90] (1,2);
\draw[line width=2.5pt,white] (0,0) to [out=90,in=-90] (2,2);
\draw[thick] (0,0) to [out=90,in=-90] (2,2);
\end{tikzpicture}
\ =\
\begin{tikzpicture}[baseline=-31,scale=-1]
\coordinate (m2) at (0.5,1);
\draw[thick] (2,0) to  [out=90,in = -90] (0,2);
\draw[line width=2.5pt,white] (1,0) to[out=90,in=-90] (m2) to[out=90,in=-90] (1,2);
\draw[thick] (1,0) to [out=90,in=-90] (m2) (m2) to [out=90,in=-90] (1,2);
\draw[line width=2.5pt,white] (0,0) to [out=90,in=-90] (2,2);
\draw[thick] (0,0) to [out=90,in=-90] (2,2);
\end{tikzpicture}.
\end{equation}
\begin{remark}
 Two relations~\eqref{eq:braiding-product} can be derived by combining relations~\eqref{graph-comp} with the `fusion relation'
 \begin{equation}\label{eq:tau-hat-graph}
\begin{tikzpicture}[scale=1.5,baseline=20,x=1pt,y=1pt]
\draw[thick] (0,13)--(0,15) (15,13)--(15,15);
\path [draw,thick]
(-5,3)  to [in=-135,out=90]  (0,13) (5,3) to [in=-45,out=90] (0,13);
\path [draw,thick]
(10,3)  to [in=-135,out=90] (15,13) (20,3) to [in=-45,out=90] (15,13);
\path [draw,thick]
(15,15) to [out=90,in=-90] (0,30);
\path [draw,thick]
(15,15) to [out=90,in=-90] (0,30);
\draw[line width=2.5pt,white] (0,15) to [out=90,in=-90] (15,30);
\path [draw,thick]
(0,15) to [out=90,in=-90] (15,30);
\path [draw,thick]
(0,15) to [out=90,in=-90] (15,30);
    \end{tikzpicture}
    \ =\
\begin{tikzpicture}[scale=1.5,baseline=20,x=1pt,y=1pt]
\coordinate (s1) at (6,19);
\path [draw,thick]  (0,28)--(0,32)  (15,28)--(15,32);
\draw[thick] (0,28) to [out=-135,in=90] (-3,23);
\draw[thick] (3,25)  to [in=-50,out=130] (0,28) (10,23) to [in=-130,out=50](15,28) ;
\draw[thick]  (15,28) to [out=-45,in=90] (18,23);
\path [draw,thick] (10,3) to [in=-90,out=90] (-3,23);
\path [draw,thick] (20,3) to [in=-50,out=90] (3,25);
\draw[line width=2.5pt,white] (10,23) to [out=-130,in=90] (-5,3);
\draw[line width=2.5pt,white] (18,23) to [out=-90,in=90] (5,3);
\path [draw,thick]  (-5,3) to [in=-130,out=90](10,23)
 ;
\path [draw,thick]  (5,3) to [out=90,in=-90] (18,23)
;
    \end{tikzpicture}\ .
\end{equation}
\end{remark}
\begin{remark}
 The braiding $\beta_{H,H}$  can be expressed entirely in terms of the structural maps, via the formula
 \begin{equation}\label{eq:braiding-str-constants}
\beta_{H,H}=(\nabla\otimes\nabla)(S\otimes (\Delta\nabla)\otimes S)(\Delta\otimes\Delta),
\end{equation}
or in graphical form:
\begin{equation}\label{tau-comp}
\beta_{H,H}=\
\begin{tikzpicture}[xscale=1,yscale=1.4,baseline=17]
\path [draw,thick]
 (1,0) to [in=-90,out=90] (0,1);
\draw[line width=3pt,white] (1,1) to [out=-90,in=90] (0,0);
\path [draw,thick]
 (0,0) to [in=-90,out=90] (1,1);
\end{tikzpicture}
 \ =\
\begin{tikzpicture}[scale=5/2,baseline=24,x=1pt,y=1pt]
\draw (0,11pt) node (x) [antipode]{} ++(11.5,0) node (y) [antipode]{};
\path [draw,thick]
 (x) to [out=90,in=-135]++(1.5,4)  (1.5,7) to [in=-90,out=135] (x)
(y) to [out=90,in=-45]++(-1.5,4) (10,7) to [in=-90,out=45]  (y);
\path [draw,thick]
 (1.5,3)  --(1.5,7)-- (5.75,9) --(5.75,9) --(5.75,13)-- (1.5,15)--(1.5,19)
(10,3)--(10,7)-- (5.75,9)
(5.75,13) --(10,15) --(10,19);
\end{tikzpicture}\ .
\end{equation}
\end{remark}
\begin{remark}
The following relation between the braiding and the antipode holds:
\begin{equation}
S\nabla=\nabla\beta_{H,H}(S\otimes S),
\end{equation}
or in graphical form
\begin{equation}\label{eq:antipode-braiding}
\begin{tikzpicture}[xscale=.6,yscale=.6,baseline=12]
\coordinate (pr) at (0.5,1);
\node (x) at (0.5,1.5) [antipode]{};
\path [draw,thick] (0,0) to [out=90,in=-135] (pr)
 (1,0) to [out=90,in=-45] (pr)--(x)--(0.5,2);
\end{tikzpicture}
\ =\
\begin{tikzpicture}[xscale=.6,yscale=12/25,baseline=6]
\coordinate (pr) at (0.5,1.5);
\node (x1) at (0,0) [antipode]{};
\node (x2) at (1,0) [antipode]{};
\draw[thick]
(0,1) to [out=-90,in=90] (x2);
\draw[line width=3pt,white] (1,1) to [out=-90,in=90] (x1);
\draw[thick] (1,1) to [out=-90,in=90] (x1);
\path [draw,thick]
  (0,-0.5)--(x1)(1,-0.5)--(x2)
(0,1) to [out=90, in=-135] (pr) (1,1) to [out=90, in=-45] (pr)--(0.5,2);
\end{tikzpicture}\ .
\end{equation}
\end{remark}

\section{Exterior algebras with deformed braided Hopf structures}\label{sec:1}
\subsection{Induced braidings on tensor and Nichols algebras}
Let $(V,\tau)$ be a braided vector space. As already mentioned in the Introduction, the tensor algebra $T(V)$ is a braided Hopf algebra in which all the elements of $V$ are primitive.  The braiding
\be
\hat\tau:=\beta_{T(V),T(V)}\colon T(V)\otimes T(V)\to T(V)\otimes T(V)
\ee is induced from $\tau$
 in the sense that  $\hat\tau\vert_{V\otimes V}=\tau$,  and it satisfies the compatibility conditions with the unit $\eta\colon\lcf\to T(V)$,
\begin{equation}
\label{Hr6}\hat\tau(\eta\otimes\operatorname{id})=\operatorname{id}\otimes\eta,\quad \hat\tau(\operatorname{id}\otimes\eta)=\eta\otimes\operatorname{id},
\end{equation}
which correspond to the graphical identities in~\eqref{graph-comp}. It also satisfies
the fusion relation for the product $\nabla\colon T(V)\otimes T(V)\to T(V)$,
\begin{equation}\label{eq:tau-hat}
\hat\tau(\nabla\otimes\nabla)=(\nabla\otimes\nabla)(\operatorname{id}\otimes \hat\tau\otimes \operatorname{id})( \hat\tau\otimes\hat\tau)(\operatorname{id}\otimes \hat\tau\otimes \operatorname{id}),
\end{equation}
which takes the graphical form shown in~\eqref{eq:tau-hat-graph}.
\begin{remark}\label{rem:pres-degree}
An immediate consequence of the fusion formula~\eqref{eq:tau-hat} is the \emph{\color{blue} preservation of degree} along the strands of the induced braiding $\hat\tau$, in the sense that
 \begin{equation}\label{eq:pres-degree}
\forall m,n\in\Z_{\ge0}\colon \hat\tau(T^m(V)\otimes T^n(V))=T^n(V)\otimes T^m(V).
\end{equation}
Since the associated Nichols algebra $\mathcal{B}(V)$ inherits the grading by degree from the tensor algebra, its corresponding braiding  also preserves the degree along the strands.
\end{remark}
\subsection{Braidings in vector spaces with linearly ordered bases}
Let $V$ be a $\lcf$-vector space, and let $\B$  be a linearly ordered basis of $V$.
Define the Heaviside theta symbol
\be\theta_{a,b}
\in \{0,1\},\quad a,b\in \B,
\ee  by setting $\theta_{a,b}=1$ if $a>b$ and $\theta_{a,b}=0$ otherwise.  The symmetric part of the Heaviside symbol is related to the Kronecker delta by
\begin{equation}\label{eq:theta-delta}
\theta_{a,b}+\theta_{b,a}=1-\delta_{a,b},\quad \forall a,b\in\B.
\end{equation}

For any scalar $p\ne0$, we define a linear map $\tau\colon V\otimes V\to V\otimes V$ by
 \begin{align}\label{eq:sl-braiding}
 \tau(a\otimes b)
& =\big(p^{\theta_{a,b}}-1\big)a\otimes b-p^{\theta_{b,a}}b\otimes a\\
&=-\theta_{a,b}\big(1-p\big)a\otimes b-p^{\theta_{b,a}}b\otimes a  \notag \\
&=
\begin{cases}
 -a\otimes a &\text{ if } a=b;\\
-(1-p)a\otimes b- b\otimes a&  \text{ if } a>b; \\
-p b\otimes a  &  \text{ if } a<b,
\end{cases} \notag
\end{align}
for all  $ a,b\in \B$. Since $\tau$ is invertible and satisfies the quantum Yang--Baxter equation~\eqref{eq:ybe} over $V$, it qualifies as an $R$-matrix. Accordingly,
 $(V,\tau)$ forms a braided $\lcf$-vector space.

\begin{remark}
 When $\operatorname{dim}(V)=N$, the braiding $\tau$ in~\eqref{eq:sl-braiding} corresponds to the $R$-matrix of the quantum group $U_q(\mathfrak{sl}_N)$  evaluated at the $N$-dimensional fundamental representation,  see, for example~\cite[Exercise 1.5(a)]{MR1470954}.
\end{remark}
\begin{lemma}\label{le:hecke}
 Let $V$ be a vector space with a linearly ordered basis $\B$. Then, the $R$-matrix $\tau$ defined in~\eqref{eq:sl-braiding},  has  eigenvalues $-1$ and $p$ with the corresponding  eigenspaces spanned  by
 \be
 \{a\otimes b+b\otimes a\mid a< b\}\cup\{a\otimes a\mid a\in\B\}
 \ee and
 \be
 \{a\otimes b-p b\otimes a\mid a< b\}.
 \ee
  It is an $R$-matrix of the Hecke type in the sense that it satisfies the quadratic equation
 \begin{equation}\label{eq:hecke-relation}
 (\tau +\operatorname{id}_V)(\tau -p\operatorname{id}_V)=0.
\end{equation}

\end{lemma}
\begin{proof}
 For any $a,b\in \B$ such that $a<b$, and a scalar $\alpha\in\lcf_{\ne0}$, consider the vector
\be w_{\alpha,a,b}:=a\otimes b+\alpha b\otimes a.
\ee
The action
\be
\tau w_{\alpha,a,b}=-pb\otimes a+\alpha(( p-1)b\otimes a-a\otimes b)=-\alpha w_{1-p+p\alpha^{-1},a,b}
\ee
implies  that $w_{1,a,b}$ and $w_{-p,a,b}$ are eigenvectors
\be
\tau w_{1,a,b}=-w_{1,a,b},\quad \tau w_{-p,a,b}=pw_{-p,a,b}
\ee
corresponding to eigenvalues $-1$ and $p$, respectively. Additionally, for any $a\in \B$, the vector $a\otimes a$ is also an eigenvector corresponding to the eigenvalue $-1$.

The quadratic equation~\eqref{eq:hecke-relation} is verified as follows:
\be
\forall a\in\B\colon  (\tau +\operatorname{id}_V)(\tau -p\operatorname{id}_V)(a\otimes a)= -(1+p)(\tau +\operatorname{id}_V)(a\otimes a)=0;
\ee
\be
\forall a,b\in\B,\ a<b\colon (\tau +\operatorname{id}_V)(\tau -p\operatorname{id}_V)(a\otimes b)=-p (\tau +\operatorname{id}_V)(a\otimes b+b\otimes a)=0;
\ee
\be
\forall a,b\in\B,\ a>b\colon (\tau +\operatorname{id}_V)(\tau -p\operatorname{id}_V)(a\otimes b)=-(\tau +\operatorname{id}_V)(a\otimes b+b\otimes a)=0.
\ee
\end{proof}
\begin{theorem}\label{thm:TV/J2}
 Let $V$ be a vector space with a linearly ordered basis $\B\subset V$, and $\tau$ the braiding on $V$ defined in~\eqref{eq:sl-braiding}. The exterior algebra $\bigwedge V$ is a braided Hopf algebra where all the elements of $V$ are primitive and the braiding is induced by $\tau$.
\end{theorem}
\begin{proof}
Any primitive element of degree two of the tensor algebra $T(V)$ is an eigenvector of $\tau$  corresponding to the eigenvalue $-1$:
\be
\Delta w=w\otimes 1+1\otimes w\quad\Rightarrow\quad \tau w=-w,\quad w\in T^2(V)\simeq V\otimes V.
\ee
This follows from the formula for the coproduct of $T(V)$ in degree two:
\be
\forall u,v\in V\colon \Delta (uv)=uv\otimes 1+1\otimes uv+u\otimes v+\tau(u\otimes v).
\ee
Lemma~\ref{le:hecke} allows us to conclude that
the primitive elements of degree two in $T(V)$  are linear combinations of  the  elements $ab+ba$ for $a,b\in \B$ with $a<b$, and $a^2$ for $a\in\B$.
This implies that the two sided ideal $\mathfrak{J}_2$ generated by these elements is a Hopf ideal, and the corresponding quotient algebra inherits the structure of a braided Hopf algebra from $T(V)$. In the quotient algebra $T(V)/\mathfrak{J}_2$, these elements vanish, and we obtain the standard presentation of the exterior algebra relative to the chosen basis $\B\subset V$:
\begin{equation}\label{q-def}
\bigwedge \!\!V\simeq T(V)/\mathfrak{J}_2=\lcf\langle \B\mid \{ba=-ab\mid a,b\in \B\}\cup\{a^2\mid a\in\B\}\rangle.
\end{equation}
\end{proof}
\begin{remark}
Notably, Theorem~\ref{thm:TV/J2} applies to both finite-dimensional and infinite-dimensional $V$.
\end{remark}
\subsection{Notation} In what follows, we denote by $\Lambda_p(V)$ the exterior algebra endowed with the braided Hopf algebra structure described in Theorem~\ref{thm:TV/J2}, and by $\hat\tau$ its associated braiding. For any $k\in\Z_{\ge0}$, we write $\Lambda_p^k(V)$ for the  subspace of $\Lambda_p(V)$ consisting of all the elements of degree $k$:
\be
\Lambda_p^k(V):=T^k(V)/(\mathcal{J}_2\cap T^k(V)),
\ee
and we denote by
\begin{equation}
\pi_k\colon \Lambda_p(V)\to\Lambda_p^k(V)
\end{equation}
the canonical projection. The corresponding graphical interpretations are as follows:
\begin{equation}
\pi_k=\
\begin{tikzpicture}[baseline=10]
 \draw[thick,mid arrow] (0,0)--(0,1) node[midway,right]{\tiny $k$};
\end{tikzpicture},\quad
\operatorname{id}_{\Lambda_p(V)}=\sum_{k\ge0}\pi_k
=\
\begin{tikzpicture}[baseline=10]
 \draw[thick,mid arrow] (0,0)--(0,1) node[midway,right]{
 };
\end{tikzpicture}.
\end{equation}
From now on, we use string diagrams with edges oriented upward,  to emphasise the forthcoming connection with the MOY diagrammatic calculus~\cite{MR1659228}.

We also denote
\begin{equation}\label{eq:proj-copr}
\Delta_{k)}:=(\pi_k\otimes\operatorname{id}_{\Lambda_p(V)})\Delta=
\begin{tikzpicture}[scale=1.65,baseline=-2]
\draw[thick,postaction={on each segment={mid arrow}} ] (0,-10pt)-- (0,0) to [out=135,in=-90] node[midway,left]{\tiny $k$}(-5pt,10pt) (0,0)
  to [out=45,in=-90] (5pt,10pt);
\end{tikzpicture},\quad \Delta_{(k}:=(\operatorname{id}_{\Lambda_p(V)}\otimes\pi_k)\Delta=
\begin{tikzpicture}[scale=1.65,baseline=-2]
\draw[thick,postaction={on each segment={mid arrow}} ] (0,-10pt)-- (0,0) to [out=135,in=-90] (-5pt,10pt) (0,0)
  to [out=45,in=-90] node[midway,right]{\tiny $k$}(5pt,10pt);
\end{tikzpicture}.
\end{equation}

We adopt the notation $|E|$ to denote the cardinality of a set $E$, which is the number of its elements when $E$ is finite.

For a set $E$, we denote its \emph{\color{blue}power set} (the set of all subsets of $E$) by
\be
\mathcal{P}(E):=\{A\mid A\subseteq E\},
\ee
and the set of all  \emph{\color{blue}finite subsets} of $E$ by
\be
 \fP(E):=\{A\mid A\subseteq E,\ |A|<\infty\}.
\ee
The set of all \emph{\color{blue}$k$-element subsets} of $E$ is denoted by
\begin{equation}\label{eq:k-subsets}
\binom{E}{k}:=\{A\subseteq E\mid |A|=k\}.
\end{equation}

For a finite linearly ordered set $E$, we denote  its minimal and maximal elements by $\min(E)$ and $\max(E)$, respectively.
Two finite subsets $E$ and $F$ of a linearly ordered set satisfy the inequality $E<F$
 if and only if $\max(E)< \min(F)$.

For convenience, in the case where $\dim(V)<\infty$, we identify the linearly ordered basis $\B\subset V$ with the set of integers $\{1,\dots,|\B|\}$, which is equipped with the natural order.

We use the following notation:
\begin{itemize}
 \item  A $q$-deformed Pochhammer symbol
 \be
(x;q)_n:=\prod_{i=0}^{n-1}(1-xq^i),\quad (x)_n:=(x;x)_n.
\ee
\item  A $q$-deformed integer
\be
[n]_q:=\frac{1-q^n}{1-q}=\sum_{i=0}^{n-1}q^i.
\ee
\item A $q$-deformed factorial
\be
[n]_q^!:=\prod_{i=1}^n[i]_q=(q)_n(q)_1^{-n}.
\ee
\item A $q$-deformed binomial coefficient
\be
\qbinom{n}{k}{q}:=\frac{[n]_q^!}{[k]_q^![n-k]_q^!}=\frac{(q)_n}{(q)_k(q)_{n-k}}.
\ee
\item A  $q$-`divided power'
 \begin{equation}\label{eq:devided-power}
 x^{\langle n\rangle_q}:=\frac{x^n}{[n]_q^!}.
\end{equation}
 \item A signed Gaussian exponential
 \begin{equation}\label{eq:gauss}
\gauss{k}:=(-1)^k p^{k(k-1)/2}.
\end{equation}
\item A quantum exponential function
\begin{equation}
\operatorname{exp}_q(x):=\sum_{n=0}^\infty x^{\langle n\rangle_q}=\frac{1}{\big( x(1-q);q\big)_\infty}.
\end{equation}
\end{itemize}

We will often write $f_{E,F,\dots,G}$ instead of $f_E\otimes f_F\otimes\dots\otimes f_G$, if the context allows usto avoid ambiguity.

\subsection{Set-theoretic bases in exterior algebras }
For any two finite subsets $A,B$ of a linearly ordered set, we extend the Heaviside theta-symbol as follows
\begin{equation}\label{eq:integer-mAB}
\theta_{A,B}=\sum_{a\in A}\sum_{b\in B}\theta_{a,b}.
\end{equation}
\begin{remark}
For any finite linearly ordered set $E$ and its subset $A\subset E$, the integer $\theta_{A,E\setminus A}$ can be characterised  by the formula
 \be
\theta_{A,E\setminus A}+\frac{|A|(|A|+1)}{2}=\sum_{k=1}^{|A|} g(k)
 \ee
 where
 \be
 g\colon \{1,2,\dots, |A|\}|\to \{1,2,\dots,|E|\}
 \ee
 is the unique strictly increasing map whose image under the order preserving identification of $E$ with $\{1,2,\dots,|E|\}$ coincides with $A$.
\end{remark}
\begin{remark}\label{rem:mAB+mBA}
 The integers $\theta_{A,B}$, defined in \eqref{eq:integer-mAB}, satisfy the symmetry relation
 \begin{equation}\label{eq:mAB+mBA}
 \theta_{A,B}+\theta_{B,A}=|A||B|-|A\cap B|
\end{equation}
 which is verified as follows:
\begin{align}
  \theta_{A,B}+\theta_{B, A}&=\sum_{i\in A}\sum_{j\in B}\theta_{i,j} +\sum_{i\in B}\sum_{j\in A}\theta_{i,j}
  =\sum_{i\in A}\sum_{j\in B}(\theta_{i,j}+\theta_{j,i})\\
  &=\sum_{i\in A}\sum_{j\in B}(1-\delta_{j,i}) =
  \sum_{i\in A}\sum_{j\in B}1-\sum_{i\in A\cap B}1=|A||B|-|A\cap B|\nonumber
\end{align}
where, in the third equality, we use the relation~\eqref{eq:theta-delta}.
\end{remark}
\begin{lemma}\label{lem:E<F}
Let $E$ and $F$ be finite subsets of $\B$ such that $E<F$. Then, the braiding $\hat\tau$ of $\Lambda_p(V)$ acts on $f_{E,F}$ as follows:
\begin{equation}
 \hat\tau f_{E,F}=(-p)^{|E||F|}f_{F,E}.
\end{equation}
\end{lemma}
\begin{proof}
 This is an easy consequence of the fusion formula~\eqref{eq:tau-hat}, \eqref{eq:tau-hat-graph}, where any elementary braiding $\tau$ is applied to a basis element $a\otimes b$ with $a\in E$ and $b\in F$, so that $a<b$, and definition~\eqref{eq:sl-braiding} implies that
 \begin{equation}
\hat\tau f_{\{a\},\{b\}}=-p f_{\{b\},\{a\}}.
\end{equation}
\end{proof}
\begin{theorem}\label{thm:1}
 Let $\B$ be a linearly ordered basis of a vector space $V$, where $p\in\lcf_{\ne0}$ a nonzero scalar, and let
 $\{f_E \mid E\in \fP(\B)\}$ be the (canonical) basis of $\bigwedge V$ given by words in the alphabet $\B$ with strictly increasing order. Then, the braided Hopf algebra  $\Lambda_p(V)$ has the following structure maps:

 the product
\begin{equation}\label{eq:prod-comb}
 \nabla(f_E\otimes f_F)=:f_Ef_F=\delta_{|E\cap F|,0}(-1)^{\theta_{E,F}}f_{E\cup F},\quad \forall E,F\in\fP(\B);
\end{equation}

 the coproduct
 \begin{equation}\label{eq:comult-comb}
 \Delta f_E=\sum_{A\subseteq E}(-p)^{\theta_{A,E\setminus A}}f_{A}\otimes f_{E\setminus A},\quad \forall E\in\fP(\B);
\end{equation}

 the antipode
 \begin{equation}\label{eq:antipode-comp}
 Sf_E=\gauss{|E|}f_E,\quad \forall E\in\fP(\B),
\end{equation}
 where the integers $\theta_{E,F}\in\Z_{\ge0}$ are defined in~\eqref{eq:integer-mAB}, and the signed Gaussian exponential in~\eqref{eq:gauss}.
 \end{theorem}
\begin{proof}
We prove~\eqref{eq:prod-comb}.

If $|E\cap F|\ne 0$, then $E\cap F\ne\emptyset$, and for any $b\in E\cap F$, we have
\be
f_Ef_F=\pm f_{E\setminus\{b\}}f_{F\setminus\{b\}}f_{\{b\}}^2=0.
\ee
From now on, we assume that $E\cap F=\emptyset$, and we proceed by induction on $|E|$.

 If $|E|=0$, then $E=\emptyset$, and the equality~\eqref{eq:prod-comb} is trivially satisfied.

Let $|E|=1$. There exists a unique decomposition $F=G\sqcup H$  such that
$
G<E<H,
$
and thus
\be
 f_Ef_F= f_Ef_G f_H=(-1)^{|G|}f_G f_E f_H=(-1)^{|G|}f_{E\cup F}.
\ee
On the other hand, if $E=\{e\}$, we have the following set-theoretic description:
\be
G=\{g\in F\mid g<e\}
\ee
which implies that $|G|=\theta_{E,F}$, and thus formula~\eqref{eq:prod-comb} holds for $|E|\le 1$.

 Let us assume that the equality~\eqref{eq:prod-comb} holds if $|E|\le k$, where $k\ge1$. Assume that $|E|=k+1$. Denoting $a:=\max(E)$, and $E':=E\setminus\{a\}$, we have
 $|E'|=k$, so that
\begin{align}
  f_Ef_F=f_{E'}f_{\{a\}}f_F&=(-1)^{\theta_{\{a\},F}}f_{E'}f_{\{a\}\cup F}
  =(-1)^{\theta_{\{a\},F}+\theta_{E',\{a\}\cup F}}f_{ E'\cup \{a\}\cup F}\\
    &=(-1)^{\theta_{\{a\},F}+\theta_{E',\{a\}\cup F}}f_{E\cup F}=
  (-1)^{\theta_{E,F}}f_{E\cup F}\nonumber
\end{align}
where, in the last equality, we used the identity
\be
\theta_{E, F}=\theta_{E\setminus\{a\},\{a\}\cup F}+\theta_{\{a\}, F},
\ee
which follows from the equality $\theta_{E\setminus\{a\},\{a\}}=0$ (recall that $a=\max(E)$).

 We prove~\eqref{eq:comult-comb} by induction on $|E|$. The formula is obviously true for $|E|=0$. Assume that it is true for all subsets of $\B$ of cardinality less than or equal to $k\ge0$. Assuming that $|E|=k+1$, denote $a:=\max(E)$ and $F:=E\setminus \{a\}$. As $|F|=k$, by the induction hypothesis, we have
 \be
  \Delta f_F=\sum_{A\subseteq F}(-p)^{\theta_{A,F\setminus A}}f_{A,F\setminus A}
 \ee
 and $f_E=f_Ff_{\{a\}}$. Using the primitivity of $f_{\{a\}}$ and Lemma~\ref{lem:E<F}, we write
\begin{align}
 \Delta f_E&=(\Delta f_F)(f_{\{\},\{a\}}+f_{\{a\},\{\}})\\
 &=\sum_{A\subseteq F}(-p)^{\theta_{A,F\setminus A}}f_{A,(F\cup\{a\})\setminus A}+\sum_{A\subseteq F}(-p)^{\theta_{A,F\setminus A}+|F|-|A|}f_{A\cup\{a\},F\setminus A}\nonumber\\
 &=\sum_{A\subseteq E\setminus\{a\}}(-p)^{\theta_{A,F\setminus A}} f_{A,E\setminus A}
 +\sum_{A\cup\{a\}\subseteq E}(-p)^{\theta_{A,F\setminus A}+|F|-|A|}
 f_{A\cup\{a\},E\setminus (A\cup\{a\})}\nonumber\\
  &=\sum_{A\subseteq E,\ a\not\in A}(-p)^{\theta_{A,F\setminus A}}f_{A,E\setminus A}
 +\sum_{A'\subseteq E,\  a\in A'}
 (-p)^{\theta_{A'\setminus\{a\},E\setminus A'}+|E|-|A'|} f_{A',E\setminus A'}.\nonumber
\end{align}
 Now, we note the identities
 \be
\theta_{A,F\setminus A}=\theta_{A,(F\cup\{a\})\setminus A}=\theta_{A,E\setminus A},\quad \forall A\subseteq F,
 \ee
 and

 \be
 \theta_{A'\setminus\{a\},E\setminus A'}+|E|-|A'|=\theta_{A',E\setminus A'},\quad A'=A\cup\{a\},\quad  \forall A\subseteq F,
 \ee
and formula~\eqref{eq:comult-comb} follows.

Now, we prove~\eqref{eq:antipode-comp}, again by induction on $|E|$. The formula is obviously true for $|E|=0$.
Assume that it is true for all subsets of $\B$ of cardinality less than or equal to $k\ge0$, and assume that $|E|=k+1$. As before, denote $a:=\max(E)$ and $F:=E\setminus \{a\}$. As $|F|=k$, by the induction hypothesis, we have
\be
Sf_F=\gauss{k}f_F.
\ee
By using the above formula, the equality
\be
S(xy)=\nabla\hat\tau((Sx)\otimes (Sy)),\quad \forall x,y\in\Lambda_p(V),
\ee
corresponding to the graphical equation~\eqref{eq:antipode-braiding}, the equality
$Sf_{\{a\}}=-f_{\{a\}}$, and Lemma~\ref{lem:E<F}, we write
\begin{align}
 Sf_E&=S(f_Ff_{\{a\}})=\nabla\hat\tau((Sf_F)\otimes (Sf_{\{a\}}))=
 -\gauss{k}\nabla \hat\tau f_{F,\{a\}}\\
 &=-\gauss{k}(-p)^{k}\nabla f_{\{a\},F}
 =-\gauss{k}(-p)^{k}f_{\{a\}}f_F
 =-\gauss{k}p^{k}f_Ff_{\{a\}}=\gauss{k+1}f_E=\gauss{|E|}f_E.\nonumber
\end{align}
\end{proof}
\begin{remark}
 The unit of $\Lambda_p(V)$ is $\eta 1=f_\emptyset$ and the counit $\epsilon f_E=\delta_{0,|E|}$.
\end{remark}
\begin{remark}
In the context of $q$-deformed exterior algebras in the representation theory of $U_q(\mathfrak{sl}_n)$ and MOY diagrams~\cite{MR1659228}, the coproduct~\eqref{eq:comult-comb} has already been considered in~\cite[Lemma~3.1.2]{MR3263166}\footnote{
We thank Francis Bonahon and L\`eo Schelstraete for pointing out these references, following the first author's talk on May 13, 2025, at the ``Hot Topics'' conference on Quantum Topology (Bonn, May 12–16, 2025).}.
\end{remark}
\begin{theorem}\label{thm:Lambda-p-Nichols}
 The exterior algebra $\Lambda_p(V)$ is isomorphic to the Nichols algebra $\mathfrak{B}(V)=T(V)/\mathfrak{J}_V$ associated with the vector space $V$ with the braiding $\tau$ defined in~\eqref{eq:sl-braiding}.
\end{theorem}
\begin{proof}
 It suffices to prove that there are no nontrivial primitive elements in  $\Lambda_p(V)$  of degrees greater than $1$. Indeed, in this case, the ideal $\mathfrak{J}_2\subset T(V)$ generated  by primitive elements of degree $2$ coincides with the maximal Hopf ideal $\mathfrak{J}_V$, which intersects $V\subset T(V)$ trivially.

 Formula~\eqref{eq:comult-comb} for the coproduct $\Delta f_E$ with $|E|\ge2$ implies that the nonprimitive contribution from the basis element $f_{A,E\setminus A}$ of $\Lambda_p(V)^{\otimes 2}$, where $A\neq\emptyset$ and $A\ne E$, does not  appear in the coproduct $\Delta f_F$ for any $F\neq E$. This means that such a term can never be cancelled by taking linear combinations of basis elements.
\end{proof}
\begin{remark}
 When $\dim(V)=2$, the braided Hopf algebra $\Lambda_p(V)$ is contained in the classification list of \cite{MR3775317}. Specifically, braiding~\eqref{eq:sl-braiding}, with $\dim(V) = 2$, fits into the braiding~\cite[$\mathfrak{R}_{2,1}$]{MR3775317}  with the substitutions $k \mapsto \sqrt{-1}$, $p \mapsto \sqrt{-1}p$, $q \mapsto \sqrt{-1}$. In this case, Theorem~\ref{thm:TV/J2}, in conjunction with  Theorem~\ref{thm:Lambda-p-Nichols}, overlaps with \cite[Proposition 3.3]{MR3775317}.
\end{remark}
\subsection{Elements of the MOY calculus}
Here, following previous works~\cite{MR1659228,MR3263166}, we introduce parts of the MOY calculus.
\begin{lemma}\label{lem:qbin}
 For any $E\in\binom{\B}{n}$, the following identity holds:
\be
\qbinom{n}{k}{p}= \sum_{A\in\binom{E}{k}}p^{\theta_{A,E\setminus A}}.\label{Lemma2}
\ee
\end{lemma}
\begin{proof} The proof follows the same line of reasoning as in the proof of~\cite[Lemma A.1]{MR1659228}.

First, we remark that for any three finite subsets $A$, $B$, and $C$ of a linearly ordered set, such that $B$ is disjoint from both $A$ and $C$, we have the obvious identity
\be
\theta_{A,B}+\theta_{A\sqcup B,C}=\theta_{A,B\sqcup C}+\theta_{B,C}\label{coassoc},
\ee
which corresponds, in particular, to the coassociativity of the coproduct.

Next, for $k=1$ and any $A\in\binom{E}{1}$, we have  $\theta_{E\setminus A,A}=n-i$ and $\theta_{A,E\setminus A}=i-1$, if $A$ is given by the $i$-th smallest element of $E$. Then
\be
 \sum_{A\in\binom{E}{1}}p^{\theta_{E\setminus A,A}}= \sum_{A\in\binom{E}{1}}p^{\theta_{A,E\setminus A}}=\sum_{i=1}^n p^{i-1}=[n]_p=\qbinom{n}{1}{p}.\label{Lemma21}
\ee

Now, we proceed by induction on $k$ by assuming that~\eqref{Lemma2} holds for $k-1\ge 1$. Inserting \eqref{Lemma21} into  the right-hand side of \eqref{Lemma2}, we write for the latter
\be\label{eq:st-ass}
 \sum_{A\in\binom{E}{k}}p^{\theta_{A,E\setminus A}}=\frac{1}{[k]_p}
  \sum_{A\in\binom{E}{k}}\sum_{B\in\binom{A}{1}}p^{\theta_{A\setminus B,B}+\theta_{A,E\setminus A}}=
\frac{1}{[k]_p} \sum_{A\in\binom{E}{k}}\sum_{B\in\binom{A}{1}}p^{\theta_{A\setminus B,E\setminus (A\setminus B)}+\theta_{B,E\setminus A}},
\ee
where, in the second equality, we use \eqref{coassoc}.
Denoting $A'= A\setminus B$, $|A'|=k-1$, and using the set-theoretic bijection
\be
\Big\{(A,B)\mid A\in\binom{E}{k},\ B\in\binom{A}{1}\Big\}\simeq \Big\{(A',B)\mid A'\in\binom{E}{k-1},\ B\in\binom{E\setminus A'}{1}\Big\},
\ee
we transform the last expression of~\eqref{eq:st-ass} into
\begin{align}
\frac{1}{[k]_p}& \sum_{A'\in\binom{E}{k-1}}\sum_{B\in\binom{E\setminus A'}{1}}
p^{\theta_{A',E\setminus A'}+\theta_{B,(E\setminus A')\setminus B}}=
\frac{[n-k+1]_p}{[k]_p} \sum_{{A'\in\binom{E}{k-1}}}
p^{\theta_{A',E\setminus A'}}\\
&=\frac{[n-k+1]_p}{[k]_p}\qbinom{n}{k-1}{p}=
\frac{[n-k+1]_p}{[k]_p}\frac{[n]_p^!}{[k-1]_p^![n-k+1]_p^!}=
\frac{[n]_p^!}{[k]_p^![n-k]_p^!}=\qbinom{n}{k}{p}.\nonumber
\end{align}
Here, in the first equality, we again use~\eqref{Lemma21} to eliminate the summation over $B$, and in the second equality, we use the induction hypothesis.
\end{proof}

\begin{lemma}\label{lem:buble} For any $k,n\in\Z_{\ge0}$, $k\le n$, we have the relations
 \begin{equation}\label{eq:buble}
 \nabla \Delta_{k)} \pi_n=\qbinom{n}{k}{p} \pi_n=
 \nabla  \Delta_{(k} \pi_n,
\end{equation}
which correspond to the following MOY diagrammatic identities:
\begin{equation}\label{eq:bubleMOY}
\begin{tikzpicture}[xscale=1.5, yscale=.9,baseline=35]
 \coordinate (n) at (0,2.5);
  \coordinate (pr) at (0,2);
   \coordinate (cpr) at (0,1);
    \coordinate (s) at (0,.5);
          \draw [thick,mid arrow](s)--(cpr) node[midway,right] {\tiny $n$} ;
     \draw[thick,mid arrow] (pr)--(n) node[midway,right] {
     } ;
     \draw[thick,mid arrow] (cpr) to [out=135, in=-135]  node[midway,left] {\tiny $k$} (pr) ;
      \draw[thick,mid arrow] (cpr) to [out=45, in=-45] (pr) ;
\end{tikzpicture}
=\qbinom{n}{k}{p}\>\>\begin{tikzpicture}[xscale=1.5, yscale=.9,baseline=35]
 \coordinate (n) at (0,2.5);
    \coordinate (s) at (0,.5);
          \draw [thick,mid arrow](s)--(n) node[midway,right] {\tiny $n$} ;
          \end{tikzpicture}
          =
          \begin{tikzpicture}[xscale=1.5, yscale=.9,baseline=35]
 \coordinate (n) at (0,2.5);
  \coordinate (pr) at (0,2);
   \coordinate (cpr) at (0,1);
    \coordinate (s) at (0,.5);
          \draw [thick,thick,mid arrow](s)--(cpr) node[midway,right] {\tiny $n$} ;
     \draw[thick,mid arrow] (pr)--(n) node[midway,right] {
     } ;
     \draw[thick,mid arrow] (cpr) to [out=135, in=-135] (pr) ;
      \draw[thick,mid arrow] (cpr) to [out=45, in=-45]  node[midway,right] {\tiny $k$} (pr) ;
\end{tikzpicture}.
    \end{equation}

\end{lemma}
\begin{proof}
Let us prove the first equality in~\eqref{eq:buble}. For any $E\in \binom{\B}{n}$, we have
\begin{align}
 \nabla \Delta_{k)} \pi_nf_E&=\nabla \Delta_{k)}f_E=\sum_{A\subseteq E} (-p)^{\theta_{A,E\setminus A}}(\pi_k f_A)f_{E\setminus A}\\
 &=\sum_{A\in \binom{E}{k}} (-p)^{\theta_{A,E\setminus A}}f_Af_{E\setminus A}
 =\sum_{A\in \binom{E}{k}} p^{\theta_{A,E\setminus A}}f_{E}=
 \qbinom{n}{k}{p}f_{E}=\qbinom{n}{k}{p}\pi_nf_{E},\nonumber
\end{align}
where, in the second-to-last equality, we use Lemma~\ref{lem:qbin}.

The proof of the second equality in~\eqref{eq:buble} is similar.
\end{proof}
\begin{remark}
Lemma~\ref{lem:buble} is analogous to~\cite[Lemma 2.2 and Lemma A.1]{MR1659228} and \cite[(2.4)]{MR3263166}.
\end{remark}
\section{A \texorpdfstring{$U_q(\mathfrak{gl}_2)$-module structure of $\Lambda_{p}(V)^{\otimes 2}$}{Lg}}\label{sec:5}
In this section, we continue to follow the works~\cite{MR1659228,MR3263166}  to adapt additional elements of the MOY calculus to our context.

We define four elements in the algebra  $\operatorname{End}(\Lambda_p(V)^{\otimes 2})$ as follows:
two invertible elements
\begin{equation}\label{eq:Tiuqsl2}
T_1=\phi_p\otimes\operatorname{id}_{\Lambda_p(V)},\quad T_2=\operatorname{id}_{\Lambda_p(V)}\otimes\phi_p,
\end{equation}
where $\phi_p\in\operatorname{Aut}(\Lambda_p(V))$ is the diagonal automorphism given by
\begin{equation}\label{eq:diag-aut-t=p}
 \phi_p=\sum_{n\ge0} p^n\pi_n;
\end{equation}
and two additional endomorphisms:
\begin{equation}\label{eq:Luqsl2}
L=(\nabla\otimes\operatorname{id}_{\Lambda_p(V)})(\operatorname{id}_{\Lambda_p(V)}\otimes \Delta_{1)}),
\end{equation}
\begin{equation}\label{eq:Ruqsl2}
R=(\operatorname{id}_{\Lambda_p(V)}\otimes\nabla)(\Delta_{(1}\otimes \operatorname{id}_{\Lambda_p(V)}),
\end{equation}
with the notation from~\eqref{eq:proj-copr}.

Using the MOY diagrammatic calculus, we have the following graphical interpretation:

\be
{ L}=\ \begin{tikzpicture}[xscale=1,yscale=0.4,baseline=13]
 \coordinate (nw) at (0,3);
  \coordinate (ne) at (1,3);
   \coordinate (sw) at (0,0);
    \coordinate (se) at (1,0);
     \coordinate (w) at (0,2);
      \coordinate (e) at (1,1);
      \draw [mid arrow, line width=0.8pt](w)--(nw);
     \draw[mid arrow, line width=0.8pt] (e)--(ne);
     \draw[mid arrow, line width=0.8pt] (e)--(w) node[midway, above] {\tiny  $1$};
      \draw[mid arrow, line width=0.8pt] (se)--(e);
      \draw[mid arrow, line width=0.8pt] (sw)--(w);
\end{tikzpicture}\ ,\quad
R=\ \begin{tikzpicture}[xscale=1,yscale=0.4,baseline=13]
 \coordinate (nw) at (0,3);
  \coordinate (ne) at (1,3);
   \coordinate (sw) at (0,0);
    \coordinate (se) at (1,0);
     \coordinate (w) at (0,1);
      \coordinate (e) at (1,2);
      \draw [mid arrow, line width=0.8pt](w)--(nw);
     \draw[mid arrow, line width=0.8pt] (e)--(ne);
     \draw[mid arrow, line width=0.8pt] (w)--(e) node[midway, above] {\tiny  $1$};
      \draw[mid arrow, line width=0.8pt] (se)--(e);
      \draw[mid arrow, line width=0.8pt] (sw)--(w);
\end{tikzpicture}\ .
\ee

\begin{proposition}
 The operators $T_1$, $T_2$, $L$, and $R$,  defined in~\eqref{eq:Tiuqsl2}--\eqref{eq:Ruqsl2}, satisfy the relations
 \begin{equation}\label{eq:xypyx}
 XY =pYX,\quad (X,Y)\in \big\{(T_1,L),(L,T_2),(T_2,R),(R,T_1)\big\},
 \end{equation}
 \begin{equation}\label{eq:t1t2t2t1}
 T_1T_2=T_2T_1,
 \end{equation}
 and
 \begin{equation}\label{eq:rllr}
 RL-LR=\frac{T_1-T_2}{1-p}.
 \end{equation}
\end{proposition}
\begin{proof}
 Relations~\eqref{eq:xypyx} and \eqref{eq:t1t2t2t1} can be verified directly using definitions~\eqref{eq:Tiuqsl2}--\eqref{eq:Ruqsl2}. For example,
 \be
\forall E,F\in\fP(\B)\colon \ T_1Lf_{E,F}=p^{1+|E|}Lf_{E,F}=pLT_1f_{E,F}\ \Rightarrow\   T_1L=pLT_1.
 \ee

 Relation~\eqref{eq:rllr} corresponds to the MOY diagrammatic identity in~\cite[Lemma A.6]{MR1659228} (see also~\cite[(3.5)]{MR3263166}). In our setting, it can be proved purely diagrammatically by using the compatibility relation of a braided Hopf algebra and the preservation of the degree along the strands of the braiding:
 \be
RL=\ \begin{tikzpicture}[yscale=.7,baseline=17]
 \coordinate (t1) at (0,2);
  \coordinate (t2) at (1,2);
   \coordinate (c1) at (0,4/3);
  \coordinate (c2) at (1,5/3);
   \coordinate (c3) at (0,2/3);
    \coordinate (c4) at (1,1/3);
     \coordinate (b1) at (0,0);
    \coordinate (b2) at (1,0);
 \path [draw,thick,postaction={on each segment={mid arrow}} ]
 (b1)-- (c3)--(c1)--(t1)(b2)--(c4)--(c2)--(t2)(c1)--(c2) node[midway,above]{\tiny $1$}
 (c4)--(c3) node[midway,below]{\tiny $1$};
\end{tikzpicture}
\ =\sum_{k=0}^1\
\begin{tikzpicture}[xscale=1.3,yscale=.7,baseline=17]
 \coordinate (t1) at (0,2);
  \coordinate (t2) at (1,2);
  \coordinate (r1) at (1,5/3);
   \coordinate (c1) at (0,4/3);
  \coordinate (c2) at (1/2,4/3);
   \coordinate (c3) at (0,2/3);
    \coordinate (c4) at (1/2,2/3);
    \coordinate (r2) at (1,1/3);
     \coordinate (b1) at (0,0);
    \coordinate (b2) at (1,0);
    \draw[thick,start arrow] (c4)--(c1);
    \draw[line width=3pt,white] (c3)--(c2);
     \draw[thick,mid end arrow] (c3)--(c2);
 \path [draw,thick,postaction={on each segment={mid arrow}} ]
 (b2)-- (r2)--(c4) node[midway,below]{\tiny $1$}(c1)--(t1)(r2)--(r1)--(t2)
 (b1)--(c3)--(c1)(c4)--(c2)node[midway,right]{\tiny $k$}--(r1)node[midway,above]{\tiny $1$};
\end{tikzpicture}
\ =\
\begin{tikzpicture}[xscale=1,yscale=.7,baseline=17]
 \coordinate (t1) at (0,2);
  \coordinate (t2) at (1,2);
   \coordinate (c1) at (0,4/3);
  \coordinate (c2) at (1,4/3);
   \coordinate (c3) at (0,2/3);
    \coordinate (c4) at (1,2/3);
     \coordinate (b1) at (0,0);
    \coordinate (b2) at (1,0);
    \draw[thick,start arrow] (c4)--(c1)node[near start,below]{\tiny $1$};
    \draw[line width=3pt,white] (c3)--(c2);
     \draw[thick,mid end arrow] (c3)--(c2)node[near end,above]{\tiny $1$};
 \path [draw,thick,postaction={on each segment={mid arrow}} ]
 (b2)--(c4)--(c2)--(t2)
 (b1)--(c3)--(c1)--(t1);
\end{tikzpicture}
\ +\
\begin{tikzpicture}[xscale=1,yscale=.7,baseline=17]
 \coordinate (t1) at (0,2);
  \coordinate (t2) at (1,2);
  \coordinate (c2) at (1,5/3);
    \coordinate (c4) at (1,1/3);
     \coordinate (b1) at (0,0);
    \coordinate (b2) at (1,0);
    \draw[thick,mid arrow] (c4) to[out=135,in=-135]node[midway,left]{\tiny $1$} (c2);
 \path [draw,thick,postaction={on each segment={mid arrow}} ]
 (b2)--(c4)--(c2)--(t2)
 (b1)--(t1);
\end{tikzpicture}\ ,
 \ee
 \be
LR=\ \begin{tikzpicture}[yscale=.7,baseline=17,xscale=-1]
 \coordinate (t1) at (0,2);
  \coordinate (t2) at (1,2);
   \coordinate (c1) at (0,4/3);
  \coordinate (c2) at (1,5/3);
   \coordinate (c3) at (0,2/3);
    \coordinate (c4) at (1,1/3);
     \coordinate (b1) at (0,0);
    \coordinate (b2) at (1,0);
 \path [draw,thick,postaction={on each segment={mid arrow}} ]
 (b1)-- (c3)--(c1)--(t1)(b2)--(c4)--(c2)--(t2)(c1)--(c2) node[midway,above]{\tiny $1$}
 (c4)--(c3) node[midway,below]{\tiny $1$};
\end{tikzpicture}
\ =\sum_{k=0}^1\
\begin{tikzpicture}[xscale=-1.3,yscale=.7,baseline=17]
 \coordinate (t1) at (0,2);
  \coordinate (t2) at (1,2);
  \coordinate (r1) at (1,5/3);
   \coordinate (c1) at (0,4/3);
  \coordinate (c2) at (1/2,4/3);
   \coordinate (c3) at (0,2/3);
    \coordinate (c4) at (1/2,2/3);
    \coordinate (r2) at (1,1/3);
     \coordinate (b1) at (0,0);
    \coordinate (b2) at (1,0);
    \draw[thick,start arrow] (c3)--(c2);
    \draw[line width=3pt,white] (c4)--(c1);
     \draw[thick,mid end arrow] (c4)--(c1);
 \path [draw,thick,postaction={on each segment={mid arrow}} ]
 (b2)-- (r2)--(c4) node[midway,below]{\tiny $1$}(c1)--(t1)(r2)--(r1)--(t2)
 (b1)--(c3)--(c1)(c4)--(c2)node[midway,left]{\tiny $k$}--(r1)node[midway,above]{\tiny $1$};
\end{tikzpicture}
\ =\
\begin{tikzpicture}[xscale=1,yscale=.7,baseline=17]
 \coordinate (t1) at (0,2);
  \coordinate (t2) at (1,2);
   \coordinate (c1) at (0,4/3);
  \coordinate (c2) at (1,4/3);
   \coordinate (c3) at (0,2/3);
    \coordinate (c4) at (1,2/3);
     \coordinate (b1) at (0,0);
    \coordinate (b2) at (1,0);
    \draw[thick,start arrow] (c4)--(c1)node[near end,above]{\tiny $1$};
    \draw[line width=3pt,white] (c3)--(c2);
     \draw[thick,mid end arrow] (c3)--(c2)node[near start,below]{\tiny $1$};
 \path [draw,thick,postaction={on each segment={mid arrow}} ]
 (b2)--(c4)--(c2)--(t2)
 (b1)--(c3)--(c1)--(t1);
\end{tikzpicture}
\ +\
\begin{tikzpicture}[xscale=-1,yscale=.7,baseline=17]
 \coordinate (t1) at (0,2);
  \coordinate (t2) at (1,2);
  \coordinate (c2) at (1,5/3);
    \coordinate (c4) at (1,1/3);
     \coordinate (b1) at (0,0);
    \coordinate (b2) at (1,0);
    \draw[thick,mid arrow] (c4) to[out=135,in=-135]node[midway,right]{\tiny $1$} (c2);
 \path [draw,thick,postaction={on each segment={mid arrow}} ]
 (b2)--(c4)--(c2)--(t2)
 (b1)--(t1);
\end{tikzpicture}\ ,
 \ee
so that
\begin{multline}
RL-LR=
\ \begin{tikzpicture}[xscale=1,yscale=.7,baseline=17]
 \coordinate (t1) at (0,2);
  \coordinate (t2) at (1,2);
  \coordinate (c2) at (1,5/3);
    \coordinate (c4) at (1,1/3);
     \coordinate (b1) at (0,0);
    \coordinate (b2) at (1,0);
    \draw[thick,mid arrow] (c4) to[out=135,in=-135]node[midway,left]{\tiny $1$} (c2);
 \path [draw,thick,postaction={on each segment={mid arrow}} ]
 (b2)--(c4)node[near start,left]{
 }--(c2)--(t2)
 (b1)--(t1)node[near start,right]{
 } ;
\end{tikzpicture}
\ -\
\begin{tikzpicture}[xscale=-1,yscale=.7,baseline=17]
 \coordinate (t1) at (0,2);
  \coordinate (t2) at (1,2);
  \coordinate (c2) at (1,5/3);
    \coordinate (c4) at (1,1/3);
     \coordinate (b1) at (0,0);
    \coordinate (b2) at (1,0);
    \draw[thick,mid arrow] (c4) to[out=135,in=-135]node[midway,right]{\tiny $1$} (c2);
 \path [draw,thick,postaction={on each segment={mid arrow}} ]
 (b2)--(c4)--(c2)--(t2) (b1)--(t1);
\end{tikzpicture}\
 =\sum_{m,n\ge0}([n]_p-[m]_p)\
\begin{tikzpicture}[xscale=1,yscale=.7,baseline=17]
 \coordinate (t1) at (0,2);
  \coordinate (t2) at (1,2);
     \coordinate (b1) at (0,0);
    \coordinate (b2) at (1,0);
 \path [draw,thick,postaction={on each segment={mid arrow}} ]
 (b2)--(t2)node[midway,right]{\tiny $n$}
 (b1)--(t1)node[midway,left]{\tiny $m$} ;
\end{tikzpicture} \\
= \sum_{m,n\ge0}\frac{p^m-p^n}{1-p}\pi_m\otimes\pi_n= \frac{T_1-T_2}{1-p}.\\
\end{multline}
\end{proof}
\begin{remark}
In the algebra generated by  $T_1$, $T_2$, $L$, and $R$,
 the elements
 \be
 C:=RL+\frac{pT_1+T_2}{(1-p)^2}=LR+\frac{T_1+pT_2}{(1-p)^2}
 \ee
 and
 \begin{equation}\label{eq:d-operator}
  D:=T_1T_2=T_2T_1
\end{equation}
generate the center.
\end{remark}
\begin{remark}
 Relations~\eqref{eq:xypyx}--\eqref{eq:rllr} can be viewed as defining a presentation of the quantum group $U_q(\mathfrak{gl}_2)$, where $q$ is related to $p$ via the identity $p=q^2$. In particular, there is an inclusion $U_q (\mathfrak{sl}_2)\subset U_q(\mathfrak{gl}_2)$ with the following identification of the generators $K$, $E$, $F$ of $U_q (\mathfrak{sl}_2)$ (see, for example, \cite[Chapter 1]{MR1359532}):
  \begin{equation}\label{eq:uqsl2-in-uqdl2}
 K= T_2D^{-1/2},\quad E= qR D^{-1/2},\quad F= L,
\end{equation}
where $D$ is defined in~\eqref{eq:d-operator}.
\end{remark}

\begin{proposition}\label{prop:LkRk}
 For any $k\in\Z_{\ge0}$, the following equalities hold:
 \be\label{eq:Lk}
L^{\langle k\rangle_p}=(\nabla\otimes\operatorname{id}_{\Lambda_p(V)})(\operatorname{id}_{\Lambda_p(V)}\otimes \Delta_{k)})=
\ \begin{tikzpicture}[xscale=1,yscale=0.4,baseline=13]
 \coordinate (nw) at (0,3);
  \coordinate (ne) at (1,3);
   \coordinate (sw) at (0,0);
    \coordinate (se) at (1,0);
     \coordinate (w) at (0,2);
      \coordinate (e) at (1,1);
      \draw [thick,mid arrow](w)--(nw);
     \draw[thick,mid arrow] (e)--(ne);
     \draw[thick,mid arrow] (e)--(w) node[midway, above] {\tiny $k$};
      \draw[thick,mid arrow] (se)--(e);
      \draw[thick,mid arrow] (sw)--(w);
\end{tikzpicture}
\ee
and
 \be\label{eq:Rk}
R^{\langle k\rangle_p}=(\operatorname{id}_{\Lambda_p(V)}\otimes\nabla)(\Delta_{(k}\otimes \operatorname{id}_{\Lambda_p(V)})=
\ \begin{tikzpicture}[xscale=1,yscale=0.4,baseline=13]
 \coordinate (nw) at (0,3);
  \coordinate (ne) at (1,3);
   \coordinate (sw) at (0,0);
    \coordinate (se) at (1,0);
     \coordinate (w) at (0,1);
      \coordinate (e) at (1,2);
      \draw [thick,mid arrow](w)--(nw);
     \draw[thick,mid arrow] (e)--(ne);
     \draw[thick,mid arrow] (w)--(e) node[midway, above] {\tiny $k$};
      \draw[thick,mid arrow] (se)--(e);
      \draw[thick,mid arrow] (sw)--(w);
\end{tikzpicture}\ .
\ee
\end{proposition}
\begin{proof}
 Let us prove~\eqref{eq:Rk} by induction on $k$. The case $k=1$ corresponds to the definition~\eqref{eq:Ruqsl2} of $R$. Assume that~\eqref{eq:Rk} holds for $k\ge1$. We proceed graphically:
 \be\label{consist}
R^{\langle k\rangle_p}R=\
\begin{tikzpicture}[xscale=1,yscale=0.4,baseline=20]
 \coordinate (nw) at (0,4);
  \coordinate (ne) at (1,4);
   \coordinate (sw) at (0,0);
    \coordinate (se) at (1,0);
     \coordinate (w1) at (0,1);
     \coordinate (w2) at (0,2);
      \coordinate (e1) at (1,2);
      \coordinate (e2) at (1,3);
\path [draw,thick,postaction={on each segment={mid  arrow}}]
(w1)--(e1) node[midway, below] {\tiny $1$} (w2)--(e2) node[midway, above] {\tiny $k$}
(se)--(e1)--(e2)--(ne) (sw)--(w1)--(w2)--(nw);
\end{tikzpicture}
\  =\
\begin{tikzpicture}[xscale=1.3,yscale=0.4,baseline=20]
 \coordinate (nw) at (0,4);
  \coordinate (ne) at (1,4);
   \coordinate (sw) at (0,0);
    \coordinate (se) at (1,0);
     \coordinate (w1) at (0,1.5);
     \coordinate (w2) at (0.33,1.83);
      \coordinate (e1) at (1,2.5);
      \coordinate (e2) at (0.66,2.16);
\path [draw,thick,postaction={on each segment={mid  arrow}}]
(w1)--(w2) node[midway, above] {
} (e2)--(e1)
(se)--(e1)--(ne) (sw)--(w1)--(nw);
\draw[thick,mid arrow] (w2) .. controls (0.37,2.5) and (0.6,2.5) .. (e2)
node[midway, above] {\tiny $k$};
\draw[thick,mid arrow] (w2) .. controls (0.37,1.5) and (0.6,1.5) .. (e2)
node[midway, below] {\tiny $1$};
\end{tikzpicture}
\  =\qbinom{k+1}{1}{p}\
\begin{tikzpicture}[xscale=1.3,yscale=0.4,baseline=20]
 \coordinate (nw) at (0,4);
  \coordinate (ne) at (1,4);
   \coordinate (sw) at (0,0);
    \coordinate (se) at (1,0);
     \coordinate (w1) at (0,1.5);
     \coordinate (w2) at (0.33,1.83);
      \coordinate (e1) at (1,2.5);
      \coordinate (e2) at (0.66,2.16);
\path [draw,thick,postaction={on each segment={mid  arrow}}]
(w1)--(e1) node[midway, above] {\tiny $k+1$}
(se)--(e1)--(ne) (sw)--(w1)--(nw) ;
\end{tikzpicture}
\  =[k+1]_{p}\
\begin{tikzpicture}[xscale=1.3,yscale=0.4,baseline=20]
 \coordinate (nw) at (0,4);
  \coordinate (ne) at (1,4);
   \coordinate (sw) at (0,0);
    \coordinate (se) at (1,0);
     \coordinate (w1) at (0,1.5);
     \coordinate (w2) at (0.33,1.83);
      \coordinate (e1) at (1,2.5);
      \coordinate (e2) at (0.66,2.16);
\path [draw,thick,postaction={on each segment={mid  arrow}}]
(w1)--(e1) node[midway, above] {\tiny $k+1$}
(se)--(e1)--(ne) (sw)--(w1)--(nw) ;
\end{tikzpicture}
\ee
so that
\be
\begin{tikzpicture}[xscale=1.3,yscale=0.4,baseline=20]
 \coordinate (nw) at (0,4);
  \coordinate (ne) at (1,4);
   \coordinate (sw) at (0,0);
    \coordinate (se) at (1,0);
     \coordinate (w1) at (0,1.5);
     \coordinate (w2) at (0.33,1.83);
      \coordinate (e1) at (1,2.5);
      \coordinate (e2) at (0.66,2.16);
\path [draw,thick,postaction={on each segment={mid  arrow}}]
(w1)--(e1) node[midway, above] {\tiny $k+1$}
(se)--(e1)--(ne) (sw)--(w1)--(nw) ;
\end{tikzpicture}
\ =\frac{R^{\langle k\rangle_p}R}{[k+1]_{p}}=\frac{R^{k+1}}{[k]_p^![k+1]_{p}}=\frac{R^{k+1}}{[k+1]_p^!}=R^{\langle k+1\rangle_p}.
\ee
The proof of~\eqref{eq:Lk} is similar.
\end{proof}

\begin{lemma}
 For all $m,n\in\Z_{\ge0}$, the following identity holds:
 \begin{equation}\label{eq:RL-to-LR}
 R^{\langle m\rangle_p}L^{\langle n\rangle_p}=\sum_{k=0}^{\min(m,n)}\frac{1}{(p)_k} L^{\langle n-k\rangle_p}R^{\langle m-k\rangle_p}\prod_{j=1}^k(T_1p^{k-m}-T_2p^{j-n}),
\end{equation}
 with the notation~\eqref{eq:devided-power}.
\end{lemma}
\begin{proof}
 This is the formula from~\cite[Lemma 1.7]{MR1359532} rewritten with identifications~\eqref{eq:uqsl2-in-uqdl2}.
\end{proof}
\begin{remark}
Identity~\eqref{eq:RL-to-LR}, restricted to $\Lambda^a_p(V)\otimes \Lambda^b_p(V)$, gives the following MOY diagrammatic identity:
 \begin{equation}\label{eq:RL-to-LR-MOY}
\begin{tikzpicture}[yscale=.4,xscale=1.3,baseline=25]
\coordinate (wn) at (0,5);
\coordinate (en) at (1,5);
\coordinate (ws) at (0,0);
\coordinate (es) at (1,0);
\coordinate (w1) at (0,1);
\coordinate (w2) at (0,2);
\coordinate (w3) at (0,3);
\coordinate (w4) at (0,4);
\coordinate (e1) at (1,1);
\coordinate (e2) at (1,2);
\coordinate (e3) at (1,3);
\coordinate (e4) at (1,4);
\draw[thick,mid arrow] (e1)--(w2) node[midway,below] {\tiny $n$} ;
\draw[thick,mid arrow] (w3)--(e4) node[midway,above] {\tiny $m$} ;
\draw[thick,mid arrow] (ws)--(w2) node[near start,left] {\tiny $a$} ;
\draw[thick,mid arrow] (w2)--(w3) ;
\draw[thick,mid arrow] (w3)--(wn) ;
\draw[thick,mid arrow] (es)--(e1) node[near start,right] {\tiny $b$} ;
\draw[thick,mid arrow] (e1)--(e4) ;
\draw[thick,mid arrow] (e4)--(en) ;
\end{tikzpicture}
=
\sum_{k=0}^{\min(m,n)}\frac{(p^{a-m+n-b};p)_k}{p^{(n-b)k}(1/p)_k}
\begin{tikzpicture}[yscale=.4,xscale=1.3,baseline=25]
\coordinate (wn) at (0,5);
\coordinate (en) at (1,5);
\coordinate (ws) at (0,0);
\coordinate (es) at (1,0);
\coordinate (w1) at (0,1);
\coordinate (w2) at (0,2);
\coordinate (w3) at (0,3);
\coordinate (w4) at (0,4);
\coordinate (e1) at (1,1);
\coordinate (e2) at (1,2);
\coordinate (e3) at (1,3);
\coordinate (e4) at (1,4);
\draw[thick,mid arrow] (w1)--(e2) node[midway,below] {\tiny $m-k$} ;
\draw[thick,mid arrow] (e3)--(w4) node[midway,above] {\tiny $n-k$} ;
\draw[thick,mid arrow] (ws)--(w1)node[near start,left] {\tiny $a$} ;
\draw[thick,mid arrow] (w1)--(w4) ;
\draw[thick,mid arrow] (w4)--(wn) ;
\draw[thick,mid arrow] (es)--(e2)  node[near start,right] {\tiny $b$};
\draw[thick,mid arrow] (e2)--(e3) ;
\draw[thick,mid arrow] (e3)--(en) ;
\end{tikzpicture}
 \end{equation}
 which is a counterpart of~\cite[Proposition A.10]{MR1659228} (see also \cite[(2.10)]{MR3263166}).
\end{remark}
\begin{remark}
 A part of MOY diagrammatic calculus is implemented by relations~\eqref{eq:bubleMOY} and \eqref{eq:RL-to-LR-MOY}, along with the MOY diagrammatic identities associated with the (braided) Hopf algebra properties of $\Lambda_p(V)$, specifically:
 \begin{equation*}
\begin{tikzpicture}[yscale=2.6,xscale=1.8,baseline=13pt]
\draw[thick,mid arrow] (0,0) to [out=90,in=-135](5pt,5pt);
\draw[thick,mid arrow] (10pt,0) to [out=90,in=-45](5pt,5pt);
\draw[thick,mid arrow] (5pt,5pt) to [out=90,in=-135](10pt,10pt);
\draw[thick,mid arrow] (20pt,0) to [out=90,in=-45](10pt,10pt);
\draw[thick,mid arrow] (10pt,10pt)--(10pt,15pt);
\end{tikzpicture}
\ =\
\begin{tikzpicture}[yscale=2.6,xscale=1.8,baseline=13pt]
\draw[thick,mid arrow]  (0,0) to [out=90,in=-135] (10pt,10pt);
\draw[thick,mid arrow]  (10pt,0) to [out=90,in=-135](15pt,5pt);
\draw[thick,mid arrow] (15pt,5pt) to [out=90,in=-45] (10pt,10pt);
\draw[thick,mid arrow]  (20pt,0) to [out=90,in=-45] (15pt,5pt);
\draw[thick,mid arrow] (10pt,10pt)--(10pt,15pt);
\end{tikzpicture} \quad(\text{associativity}),
\qquad
\begin{tikzpicture}[yscale=3.9,xscale=3,baseline=14.3pt]
\node (x) [unit]{};
\draw[thick,mid arrow]  (x) to [out=90,in=-135] node[midway,above]{\tiny $0$} +(5pt,5pt);
\draw[thick,mid arrow]  (x)+(10pt,0) to [out=90,in=-45] ++(5pt,5pt);
\draw[thick,mid arrow]  (x)+(5pt,5pt)--(5pt,10pt);
\end{tikzpicture}
\ =\
\begin{tikzpicture}[yscale=3.9,xscale=3,baseline=14.3pt]
\draw[thick,mid arrow] (0,0)--(0,10pt);
\end{tikzpicture}
\ =\
\begin{tikzpicture}[yscale=3.9,xscale=3,baseline=14.3pt]
\node (x) [unit]{};
\draw[thick,mid arrow]  (x) to [out=90,in=-45] node[midway,above]{\tiny $0$}  +(-5pt,5pt);
\draw[thick,mid arrow]  (x)+(-10pt,0) to [out=90,in=-135] ++(-5pt,5pt);
\draw[thick,mid arrow]  (x)+(-5pt,5pt)--(-5pt,10pt);
\end{tikzpicture}
\quad (\text{unitality}),
\end{equation*}
\begin{equation*}
\begin{tikzpicture}[yscale=2.6,xscale=1.8,baseline=13pt]
\draw[thick,mid arrow]  (0,0)--(0,5pt);
\draw[thick,mid arrow] (0,5pt) to [out=135,in=-90]  (-5pt,10pt);
\draw[thick,mid arrow]   (-5pt,10pt) to [out=135,in=-90] node[near start,above]{
}(-10pt,15pt);
\draw[thick,mid arrow] (-5pt,10pt) to [out=45,in=-90] (0,15pt);
\draw[thick,mid arrow]  (0,5pt) to [out=45,in=-90] node[midway,above]{
}(10pt,15pt);
\end{tikzpicture}
\ =\
\begin{tikzpicture}[yscale=2.6,xscale=1.8,baseline=13pt]
\draw[thick,mid arrow]  (0,0)--(0,5pt);
\draw[thick,mid arrow] (0,5pt) to [out=45,in=-90] (5pt,10pt);
\draw[thick,mid arrow] (5pt,10pt) to [out=45,in=-90] node[near start,above]{
}(10pt,15pt);
\draw[thick,mid arrow] (5pt,10pt) to [out=135,in=-90] (0,15pt);
\draw[thick,mid arrow]  (0,5pt) to [out=135,in=-90] node[midway,above]{
} (-10pt,15pt);
\end{tikzpicture}
\quad (\text{coassociativity}),
\qquad
\begin{tikzpicture}[yscale=3.9,xscale=3,baseline=-25]
\node (x) [counit]{};
\draw[thick,mid arrow] (x)+(5pt,-5pt) to [out=135,in=-90] node[near end,below]{\tiny $0$}  (x);
\draw[thick,mid arrow] (x)+(5pt,-5pt) to [out=45,in=-90] +(10pt,0);
\draw[thick,mid arrow](x)+(5pt,-10pt)--+(5pt,-5pt);
\end{tikzpicture}
\ =\
\begin{tikzpicture}[yscale=3.9,xscale=3,baseline=13]
\draw[thick,mid arrow] (0,0)--(0,10pt);
\end{tikzpicture}
\ =\
\begin{tikzpicture}[yscale=3.9,xscale=3,baseline=-25]
\node (x) [counit]{};
\draw[thick,mid arrow] (x)+(-5pt,-5pt) to [out=45,in=-90] node[near end,below]{\tiny $0$} (x) ;
\draw[thick,mid arrow] (x)+(-5pt,-5pt) to [out=135,in=-90] +(-10pt,0);
\draw[thick,mid arrow] (x)+(-5pt,-10pt)--+(-5pt,-5pt);
\end{tikzpicture}\quad
(\text{counitality}),
\end{equation*}
etc.
\end{remark}

\begin{lemma}\label{Lemma5}
For any $i,j\in\Z_{\ge0}$, the following summation formula holds
\be\label{eq:sum-form}
\sum_{k=0}^{i}\sum_{l=0}^{j}\qbinom{i}{k}{p}\frac{(p^{i-j-k};p)_{l}}{(p)_{l}}\gauss{k}p^{l}=\delta_{i,j}
\ee
\end{lemma}
\begin{proof} We recall the $q$-binomial formula
\begin{equation}\label{eq:qbinom}
(x;p)_n=\sum_{m=0}^n\qbinom{n}{m}{p}\gauss{m}x^m,\quad \forall n\in\Z_{\ge0},
\end{equation}
and the summation formula
\begin{equation}\label{eq:summation-formula}
\sum_{m=0}^n\frac{p^m}{(p)_m}=\frac1{(p)_n},\quad \forall n\in\Z_{\ge0}.
\end{equation}

Denoting by $z_{i,j}$ the left-hand side of~\eqref{eq:sum-form},  and using~\eqref{eq:qbinom} twice --- first to expand the factor $(p^{i-j-k};p)_{l}$, and then to perform the summation over $k$ --- we obtain
\be
 z_{i,j}=\sum_{l=0}^j\frac{p^l}{(p)_{l}}\sum_{s=0}^l \qbinom{l}{s}{p}\gauss{s}\sum_{k=0}^{i} \gauss{k}\qbinom{i}{k}{p}p^{(i-j-k)s}=\sum_{l=0}^j\frac{p^l}{(p)_{l}}\sum_{s=0}^l \gauss{s}\qbinom{l}{s}{p}(p^{-s};p)_ip^{(i-j)s}
\ee
--- We then exchange the order of summation, shift the variable $l\mapsto l+s$, apply~\eqref{eq:summation-formula}, and use the fact that $(p^{-s};p)_i\ne 0$ only if $s\ge i$. In this way, we obtain  ---
\be
 =\sum_{s=0}^j\gauss{s}\frac{(p^{-s};p)_i}{(p)_{s}}p^{(i-j+1)s}\sum_{l=0}^{j-s} \frac{p^{l}}{(p)_{l}}
  =\sum_{s=i}^j\gauss{s}\frac{(p^{-s};p)_i}{(p)_{s}(p)_{j-s}}p^{(i-j+1)s}
\ee
  --- where the last expression trivially vanishes unless $j\ge i$; from this point on, we assume that this condition holds.
  Finally, we shift the variable $s\mapsto s+i$, use the identity $(p^{-s-i};p)_i=\frac{(1/p)_{s+i}}{(1/p)_{s}}$, simplify, and apply the $q$-binomial formula~\eqref{eq:qbinom} to remove the summation ---
\begin{align}  &=\sum_{s=0}^{j-i}\gauss{s+i}\frac{(p^{-s-i};p)_i}{(p)_{s+i}(p)_{j-i-s}}p^{(i-j+1)(s+i)}
  =\sum_{s=0}^{j-i}\gauss{s+i}\frac{(1/p)_{s+i}}{(p)_{s+i}(p)_{j-i-s}(1/p)_s}p^{(i-j+1)(s+i)} \\
   &=\sum_{s=0}^{j-i}\frac{1}{(p)_{j-i-s}(1/p)_s}p^{(i-j)(s+i)}   =p^{(i-j)i} \sum_{s=0}^{j-i}\frac{\gauss{s}}{(p)_{j-i-s}(p)_s}p^{(i-j+1)s}\nonumber \\
    &=\frac{p^{(i-j)i} }{(p)_{j-i}}\sum_{s=0}^{j-i}\qbinom{j-i}{s}{p}\gauss{s}p^{(i-j+1)s}  =\frac{p^{(i-j)i} }{(p)_{j-i}}(p^{i-j+1};p)_{j-i} =\delta_{i,j}.\nonumber
 \end{align}
\end{proof}

\begin{proposition}\label{prop:braiding-MOY}
 The action of the braiding $\hat\tau$ of $\Lambda_p(V)$ is of the form
 \begin{equation}
 \hat\tau f_{E,F}=\sum_{k=0}^{\min(|E|,|F|)}\gauss{k}L^{\langle |F|-k\rangle_p}R^{\langle |E|-k\rangle_p} f_{E,F},
\end{equation}
 where we use notations~\eqref{eq:devided-power} and \eqref{eq:gauss}, which correspond to the MOY diagrammatic identity
\begin{equation}\label{MOYtau}
\hat\tau(\pi_m\otimes\pi_n)=
\begin{tikzpicture}[yscale=1.5,baseline=20]
 \draw[thick,mid end arrow] (1,0) to [out=90,in=-90] node[near start,right] {\tiny  $n$} (0,1);
\draw[thick,line width=3pt,white] (0,0) to [out=90,in=-90] (1,1);
\draw[thick,mid end arrow] (0,0) to [out=90,in=-90] node[near start,left] {\tiny  $m$}  (1,1);
\end{tikzpicture}
=
\sum_{k=0}^{\min(m,n)}\gauss{k}
\begin{tikzpicture}[yscale=.32,xscale=1.3,baseline=13]
\coordinate (wn) at (0,4);
\coordinate (en) at (1,4);
\coordinate (ws) at (0,-1);
\coordinate (es) at (1,-1);
\coordinate (w1) at (0,0);
\coordinate (w2) at (0,1);
\coordinate (w3) at (0,2);
\coordinate (w4) at (0,3);
\coordinate (e1) at (1,0);
\coordinate (e2) at (1,1);
\coordinate (e3) at (1,2);
\coordinate (e4) at (1,3);
\draw[thick,mid arrow] (w1)--(e2)  ;
\draw[thick,mid arrow] (e3)--(w4) node[midway,above] {\tiny  $n-k$}
;
\draw[thick,mid arrow] (ws)--(w1) node[midway,left] {\tiny  $m$} ;
\draw[thick,mid arrow] (w1)--(w4)  node[midway,left] {\tiny  $k$} ;
\draw[thick,mid arrow] (w4)--(wn)node[midway,left] {
} ;
\draw[thick,mid arrow] (es)--(e2)  node[midway,right] {\tiny  $n$};
\draw[thick,mid arrow] (e2)--(e3) ;
\draw[thick,mid arrow] (e3)--(en)  node[midway,right] {
} ;
\end{tikzpicture}.
\end{equation}
\end{proposition}
\begin{proof}
Composing  \eqref{eq:braiding-str-constants} with $\pi_m\otimes\pi_n$
and using the MOY diagrammatic language, we obtain
\begin{equation}\label{tauLR1}
\begin{tikzpicture}[yscale=1.5,baseline=20]
 \draw[thick,mid end arrow] (1,0) to [out=90,in=-90] node[near start,right] {\tiny  $n$} (0,1);
\draw[thick,line width=3pt,white] (0,0) to [out=90,in=-90] (1,1);
\draw[thick,mid end arrow] (0,0) to [out=90,in=-90] node[near start,left] {\tiny  $m$}  (1,1);
\end{tikzpicture}=
\begin{tikzpicture}[scale=1,baseline=13,x=1pt,y=1pt]
\draw[thick,mid arrow] (0,-10) -- (0,0)  node[midway,left] {\tiny $m$};
\draw[thick] (0,0) -- (0,13.5) ;
\draw[thick,mid arrow] (0,16.5) -- (0,30) ;
\draw[thick] (40,0) -- (40,13.5) ;
\draw[thick,mid arrow] (40,16.5) -- (40,30) ;
\draw[thick,mid arrow] (0,0) -- (20,10);
\draw[thick,mid arrow] (20,10) -- (20,20);
\draw[thick,mid arrow] (20,20) -- (0,30);
\draw[thick,mid arrow] (0,30) -- (0,40);
\draw[thick,mid arrow] (40,-10) -- (40,0)  node[midway,right] {\tiny $n$};
\draw[thick,mid arrow] (40,0) -- (20,10);
\draw[thick,mid arrow] (20,20) -- (40,30);
\draw[thick,mid arrow] (40,30) -- (40,40) ;
\draw (0,15) node (x) [antipode]{} ++(40,0) node (y) [antipode]{};
\end{tikzpicture}
=
\sum_{s=-m}^{n}\sum_{k=0}^{\min(m,n-s)} \sum_{l=0}^{\min(n,m+s)} \gauss{k}\gauss{l}
\begin{tikzpicture}[scale=1,baseline=13,x=1pt,y=1pt]
\draw[thick,mid arrow] (0,-10) -- (0,0)  node[midway,left] {\tiny $m$};
\draw[thick,mid arrow] (0,0) -- (0,30) node[midway,left] {\tiny $k$
};
\draw[thick,mid arrow] (0,0) -- (20,10)node[midway,above] {
};
\draw[thick,mid arrow] (20,10) -- (20,20);
\draw[thick,mid arrow] (20,20) -- (0,30);
\draw[thick,mid arrow] (0,30) -- (0,40) node[midway,left] {
};
\draw[thick,mid arrow] (40,-10) -- (40,0)  node[midway,right] {\tiny $n$};
\draw[thick,mid arrow] (40,0) -- (40,30) node[midway,right] {\tiny $l$
};
\draw[thick,mid arrow] (40,0) -- (20,10)node[midway,above] {
};
\draw[thick,mid arrow] (20,20) -- (40,30)node[midway,above] {
};
\draw[thick,mid arrow] (40,30) -- (40,40) node[midway,right] {\tiny $m+s$
};
\end{tikzpicture},
\ee
where the factors $\gauss{k}$ and $\gauss{l}$ come from the action of the antipode \eqref{eq:antipode-comp}.
Deforming the last diagram, we can use \eqref{eq:RL-to-LR-MOY}:
\be
\begin{tikzpicture}[xscale=1.5,baseline=13,x=1pt,y=1pt]
\draw[thick,mid arrow] (0,-10) -- (0,0)  node[near start,left] {\tiny $m$};
\draw[thick,mid arrow] (0,0) -- (0,30) ;
\draw[thick,mid arrow] (0,0) to [out=150, in =-150]  node[midway,left] {\tiny  $k$}(0,30) ;
\draw[thick,mid arrow] (0,22) -- (35,30) node[midway,above] {\tiny $m+s-l$};
\draw[thick,mid arrow] (0,30) -- (0,40) node[near end,left] {
};
\draw[thick,mid arrow] (35,-10) -- (35,0)  node[near start,right] {\tiny $n$};
\draw[thick,mid arrow] (35,0) -- (35,30) node[midway,right] {
} ;
\draw[thick,mid arrow] (35,0) -- (0,8) node[midway,below] {\tiny $n-l$} ;
\draw[thick,mid arrow] (35,30) -- (35,40) node[near end,right] {
};
\end{tikzpicture}
 =
\sum_{j=l}^{\min(n,m+s)}\frac{(p^{-k-s};p)_{j-l}}{(1/p)_{j-l}}p^{l(j-l)}\
\begin{tikzpicture}[xscale=1.5,baseline=13,x=1pt,y=1pt]
\draw[thick,mid arrow] (0,-10) -- (0,-3)  node[near start,left] {\tiny $m$};
\draw[thick,mid arrow] (0,-3) -- (0,33) ;
\draw[thick,mid arrow] (0,-3) to [out=150, in =-150] node[midway,left] {\tiny  $k$}(0,33) ;
\draw[thick,mid arrow]  (35,22) -- (0,30)  node[midway,above] {\tiny $n-j$
};
\draw[thick,mid arrow] (0,33) -- (0,40) node[near end,left] {
};
\draw[thick,mid arrow] (35,-10) -- (35,0)  node[near start,right] {\tiny $n$};
\draw[thick,mid arrow] (35,0) -- (35,30)  ;
\draw[thick,mid arrow] (0,0) -- (35,8) node[midway,below] {\tiny $m+s-j$} ;
\draw[thick,mid arrow] (35,30) -- (35,40) node[near end,right] {
};
\end{tikzpicture}.
\end{equation}
Now, using associativity, coassociativity and \eqref{eq:bubleMOY}, we  obtain
\be
\begin{tikzpicture}[xscale=1.5,baseline=13,x=1pt,y=1pt]
\draw[thick,mid arrow] (0,-10) -- (0,-3)  node[near start,left] {\tiny $m$};
\draw[thick,mid arrow] (0,-3) -- (0,33) ;
\draw[thick,mid arrow] (0,-3) to [out=150, in =-150] node[midway,left] {\tiny $k$}(0,33) ;
\draw[thick,mid arrow]  (35,22) -- (0,30) node[midway,above] {\tiny $n-j$};
\draw[thick,mid arrow] (0,33) -- (0,40) node[near end,left] {
};
\draw[thick,mid arrow] (35,-10) -- (35,0)  node[near start,right] {\tiny $n$};
\draw[thick,mid arrow] (35,0) -- (35,30)  ;
\draw[thick,mid arrow] (0,0) -- (35,8) node[midway,below] {\tiny $m+s-j$} ;
\draw[thick,mid arrow] (35,30) -- (35,40) node[near end,right] {
};
\end{tikzpicture}
=
\begin{tikzpicture}[xscale=1.5,baseline=13,x=1pt,y=1pt]
\draw[thick,mid arrow] (0,-10) -- (0,0)  node[near start,left] {\tiny $m$};
          \draw [thick,mid arrow](0,0)--(0,8) node[midway,left] {
          } ;
          \draw [thick,mid arrow](0,22)--(0,30) node[midway,left] {
          } ;
    \draw[thick,mid arrow] (0,8) to [out=150, in=-150]  node[midway,left] {\tiny $k$} (0,22);
      \draw[thick,mid arrow] (0,8) to [out=30, in=-30] (0,22) ;
      \draw[thick,mid arrow]  (35,22) -- (0,30) node[midway,above] {\tiny $n-j$};
\draw[thick,mid arrow] (0,30) -- (0,40) node[near end,left] {
};
\draw[thick,mid arrow] (35,-10) -- (35,0)  node[near start,right] {\tiny $n$};
\draw[thick,mid arrow] (35,0) -- (35,30)  ;
\draw[thick,mid arrow] (0,0) -- (35,8) node[midway,below] {\tiny $m+s-j$} ;
\draw[thick,mid arrow] (35,30) -- (35,40) node[near end,right] {
};
\end{tikzpicture}
 =
\qbinom{j-s}{k}{p}
\begin{tikzpicture}[xscale=1.5,baseline=13,x=1pt,y=1pt]
\draw[thick,mid arrow] (0,-10) -- (0,-3)  node[near start,left] {\tiny $m$};
\draw[thick,mid arrow] (0,-3) -- (0,33) node[midway,right] {
};
\draw[thick,mid arrow]  (35,22) -- (0,30) node[midway,above] {\tiny $n-j$};
\draw[thick,mid arrow] (0,33) -- (0,40) node[near end,left] {
};
\draw[thick,mid arrow] (35,-10) -- (35,0)  node[near start,right] {\tiny $n$};
\draw[thick,mid arrow] (35,0) -- (35,30)  ;
\draw[thick,mid arrow] (0,0) -- (35,8) node[midway,below] {\tiny $m+s-j$} ;
\draw[thick,mid arrow] (35,30) -- (35,40) node[near end,right] {
};
\end{tikzpicture}.
\ee
Putting everything together, changing the variables $s\mapsto j-i$,  $l\mapsto j-l$, and reordering the summations, we obtain
\begin{equation}
\begin{tikzpicture}[yscale=1.5,baseline=20]
 \draw[thick,mid end arrow] (1,0) to [out=90,in=-90] node[near start,right] {\tiny  $n$} (0,1);
\draw[thick,line width=3pt,white] (0,0) to [out=90,in=-90] (1,1);
\draw[thick,mid end arrow] (0,0) to [out=90,in=-90] node[near start,left] {\tiny  $m$}  (1,1);
\end{tikzpicture}
= \sum_{i=0}^{m}\sum_{j=0}^{n}\gauss{j}z_{i,j}
\begin{tikzpicture}[xscale=1,baseline=13,x=1pt,y=1pt]
\draw[thick,mid arrow] (0,-10) -- (0,-3)  node[near start,left] {\tiny $m$};
\draw[thick,mid arrow] (0,-3) -- (0,33) node[midway,right] {
};
\draw[thick,mid arrow]  (35,22) -- (0,30) node[midway,above] {\tiny $n-j$};
\draw[thick,mid arrow] (0,33) -- (0,40) node[near end,left] {
};
\draw[thick,mid arrow] (35,-10) -- (35,0)  node[near start,right] {\tiny $n$};
\draw[thick,mid arrow] (35,0) -- (35,30)  ;
\draw[thick,mid arrow] (0,0) -- (35,8) node[midway,below] {\tiny $m-i$} ;
\draw[thick,mid arrow] (35,30) -- (35,40) node[near end,right] {
};
\end{tikzpicture},
\end{equation}
where
\be
z_{i,j}:=\sum_{k=0}^{i}\sum_{l=0}^{j}\qbinom{i}{k}{p}\frac{(p^{i-j-k};p)_{l}}{(p)_{l}}\gauss{k}p^{l}.
\ee
Finally, using { Lemma \ref{Lemma5}}, we obtain formula~\eqref{MOYtau}.
\end{proof}
\begin{remark}
 Formula~\eqref{MOYtau} for the braiding $\hat\tau$ reflects the  preservation of the degree along the strands, in accordance with the general property of tensor and Nichols algebras stated in Remark~\ref{rem:pres-degree}.
\end{remark}
\begin{remark}
  Formula~\eqref{MOYtau} is a variant  of the one given in~\cite[Theorem 5.1]{MR1659228} and~\cite[Corollary 6.2.3]{MR3263166}.
\end{remark}

\section{Matrix coefficients of the braiding of \texorpdfstring{$\Lambda_p(V)$}{Lg}}\label{sec:2}

In this section we explicitly calculate the matrix elements of the braiding matrix in the set-theoretic basis.
First, we introduce a different representation for $\hat\tau$ based on an analogy with particle
physics.

The braiding  of  $\Lambda_p(V)$ can be decomposed into a sum
\be
\hat\tau=\sum_{k\in\Z_{\ge0}}\tau_k
\ee
where $\tau_k$ acts on the set-theoretic basis of $\Lambda_p(V)^{\otimes 2}$ as follows
\begin{equation}\label{eq:s-coeff}
\tau_kf_{E,F}=(-1)^{|E||F|}\sum s_{E,G;F,H} f_{F',E'}
\end{equation}
where
\be
E':=(E\setminus G)\cup H,\quad F':=(F\setminus H)\cup G,
\ee
 and the summation runs over the subsets $G\in \binom{E\setminus F}{k}$ and $H\in \binom{F\setminus E}{k}$.
In particular, the conditions on $G$ and $H$ imply the equalities
\begin{equation}\label{eq:balance1}
|E'|=|E|,\quad |F'|=|F|
\end{equation}
in agreement with Remark~\ref{rem:pres-degree}.

The simplest case of~\eqref{eq:s-coeff} corresponds to $k=0$ with the only option $G=H=\emptyset$:
\begin{equation}\label{eq:0-channel}
\tau_0f_{E,F}=(-1)^{|E||F|}s_{E;F} f_{F,E},\quad s_{E;F}:=s_{E,\emptyset;F,\emptyset}.
\end{equation}
Adopting terminology from particle physics, we refer to $\tau_0$ as the `(elastic) scattering channel' of the braiding $\hat\tau$, and to $\tau_k$ as the `size-$k$ exchange channel'.

Before we proceed to calculation of the coefficients $s_{E,G;F,H}$, let us prove two useful lemmas.

\begin{lemma}\label{lem:sum-to-product}
Let $E$ be a finite set and let $R$ be a commutative unital ring. Let $g\colon \mathcal{P}(E)\to R$ be a multiplicative map, meaning that for any disjoint subsets $A,B\in \mathcal{P}(E)$, we have $g(A\sqcup B)=g(A)g(B)$. Then the  following identity holds
 \begin{equation}\label{eq:multiplicative-map}
 \sum_{A\subseteq E}g(A)=g(\emptyset)\prod_{A\in\binom{E}{1}}(1+g(A)).
\end{equation}
\end{lemma}
\begin{proof}
We proceed by induction on $|E|$. For $|E|=0$, identity~\eqref{eq:multiplicative-map} obviously holds. Assume that it holds for all sets of cardinality $k\ge0$, and let $|E|=k+1$. Fix an element $x\in E$, and set $F:=E\setminus\{x\}$. Then $|F|=k$, so  by the induction hypothesis, we have
\be
 \sum_{A\subseteq F}g(A)=g(\emptyset)\prod_{A\in\binom{F}{1}}(1+g(A)).
\ee
 On the other hand, any subset of $E$ either contains $x$ or does not, so
\begin{align}
  \sum_{A\subseteq E}g(A)&=\sum_{A\subseteq F}g(A)+\sum_{A\subseteq F}g(A\sqcup\{x\})=\sum_{A\subseteq F}g(A)+\sum_{A\subseteq F}g(A)g(\{x\})\\
  &=(1+g(\{x\}))\sum_{A\subseteq F}g(A)=(1+g(\{x\}))g(\emptyset)\prod_{A\in\binom{F}{1}}(1+g(A))=g(\emptyset)\prod_{A\in\binom{E}{1}}(1+g(A)).
  \nonumber
\end{align}
\end{proof}

\begin{lemma}\label{lem:alpha0}
 Let $G$ et $H$ be two finite subsets of a linearly ordered set, and assume that $p$ is not a root of unity. Then, the product
 \begin{equation}\label{eq:uGH}
 \alpha_{G,H}:=\prod_{A\in\binom{G}{1}}(p^{\theta_{A,H}-\theta_{A,G}}-1)
\end{equation}
does not vanish if and only if  $|H|\ge |G|$, and the map
 \be
 \theta_{\cdot,H}\colon \binom{G}{1}\to \Z_{\ge 0}
 \ee
 is injective and its image is contained in $\Z_{> 0}$.
  \end{lemma}
\begin{proof} When $|G|=0$, we have $G=\emptyset$, so $\alpha_{G,H}=1$. Thus, there are no conditions on $H$, and the statement  holds.

Now, assuming that $|G|>0$, we prove the `only if' part by induction on the size of $G$. Assume that  $ \alpha_{G,H}\ne0$. Suppose that the statement holds for $G':=G\setminus\{g\}$, where $g:=\min(G)$. In this case, we have $\theta_{\{g\},G}=0$, so that
\begin{equation}\label{eq:alpha-rec}
\alpha_{G,H}=(p^{\theta_{\{g\},H}}-1)v_{G',H}
\end{equation}
where
\be
v_{G',H}:=\prod_{A\in\binom{G'}{1}}(p^{\theta_{A,H}-1-\theta_{A,G'}}-1).
\ee
The condition $\alpha_{G,H}\ne0$ implies $\theta_{\{g\},H}>0$. This means $H\neq\emptyset$ and $g>h:=\min(H)$. Thus, for every $A\in\binom{G'}{1}$, we have $\theta_{A,H}=1+\theta_{A,H'}$, where $H':=H\setminus\{h\}$. Therefore
$0\ne v_{G',H}=\alpha_{G',H'}$. This completes  the induction on $|G|$ for the `only if' part.

Now, assuming $|G|>0$, we prove the `if' part. Let $G=\{g_1,g_2,\dots, g_m\}$ where $g_i$ is the $i$-th smallest element of $G$. Then, for any $i\in\{1,\dots,m\}$, we have $\theta_{\{g_i\},G}=i-1$, so that
\begin{equation}
\alpha_{G,H}=(p^{\theta_{\{g_1\},H}}-1)(p^{\theta_{\{g_2\},H}-1}-1)\cdots (p^{\theta_{\{g_m\},H}-m+1}-1).
\end{equation}
By assumption, $\theta_{\{g_1\},H}>0$, and the sequence $(\theta_{\{g_i\},H})_{i=1}^m$ is strictly increasing. This means that, for any  $i\in\{1,\dots,m\}$, $\theta_{\{g_i\},H}>i-1$, so that
$\alpha_{G,H}\ne 0$.
\end{proof}

Consider the operator
\begin{equation}\label{eq:B-operator}
B:= (\nabla\otimes\operatorname{id}_{\Lambda_p(V)})(S\otimes(\Delta\nabla)) (\Delta\otimes\operatorname{id}_{\Lambda_p(V)})=\operatorname{exp}_p(L)(S\otimes \operatorname{id}_{\Lambda_p(V)})\operatorname{exp}_p(R),
\end{equation}
or in the graphical form
\be B=\
\begin{tikzpicture}[baseline=19,yscale=.75]
 \coordinate (t1) at (0,2);
  \coordinate (t2) at (1,2);
   \coordinate (c1) at (0,3/2);
  \coordinate (c2) at (1,4/3);
   \coordinate (c3) at (0,1/2);
    \coordinate (c4) at (1,2/3);
     \coordinate (b1) at (0,0);
    \coordinate (b2) at (1,0);
    \node (a) at (0,1)[antipode]{};
 \path [draw,thick,postaction={on each segment={mid arrow}} ]
 (b1)-- (c3)--(a)--(c1)--(t1)(b2)--(c4)--(c2)--(t2)(c3)--(c4) (c2)--(c1);
\end{tikzpicture}\ .
\ee

For any finite subsets $E$ and $F$ of $\B$, and any $G$ and $H$ such that
\be \label{eq:G-and-H}
G\subseteq E\setminus F, \quad H\subseteq F\setminus E,
\ee
we define
\be\label{eq:dotE-dotF-C}
C:=E\cap F, \quad  \dot E:=(E\setminus F)\setminus G,\quad \dot F:=(F\setminus E)\setminus H
\ee
and
\be\label{eq:Ep-and-Fp}
E':=(E\setminus G)\cup H,\quad F':=(F\setminus H)\cup G.
\ee
This can be visualised via the following Venn diagram
\be\label{Venn}
 \begin{tikzpicture}[xscale=2.5,yscale=2,baseline=-1,x=1pt,y=1pt]
\def\CA{(0,0) ellipse (20 and 10)}
\def\CB{(22,0) ellipse (17 and 12)}
\draw[line width=0.7,fill=yellow!10] \CB node at (25,16) {$F$};
          \begin{scope}
            \clip \CA;\fill[black,opacity=0.5] \CB ;
        \end{scope}
\draw[line width=0.7,fill=green!20,opacity=0.6] \CA node at (-5,14) {$E$};
\draw[fill=blue!25,dashed] (-8,-1) circle [radius=4] node {$G$}  ;
\draw[fill=red!25,dashed] (28,-3) circle [radius=5] node {$H$};
\node at (12,0) {$C$};\node at (0,5) {$\dot E$};\node at (23,7) {$\dot F$};
    \end{tikzpicture}\ .
\ee
\begin{lemma}\label{lem:betaEGFH}
Let coefficients $\beta_{E,G;F,H}$ be defined through the formula
\begin{equation}
Bf_{E,F}=(-1)^{|E||F|}\sum_{G,H}\beta_{E,G;F,H}f_{F',E'},
\end{equation}
where operator $B$ is defined in~\eqref{eq:B-operator}, $G$ and $H$ satisfy~\eqref{eq:G-and-H}, and  $E'$ and $F'$ are given in~\eqref{eq:Ep-and-Fp}. Then,
\begin{equation}\label{eq:betaEGFH}
\beta_{E,G;F,H}= (-1)^{\theta_{F,E}+\theta_{F',E'}} p^{\theta_{G\sqcup C,E}+\theta_{\dot F,E'}}\alpha_{G,H},
\end{equation}
where $C$ and $\dot F$ are defined in~\eqref{eq:dotE-dotF-C} and
$\alpha_{G,H}$ is defined in~\eqref{eq:uGH}.
\end{lemma}
\begin{proof}
The contributions to $\beta_{E,G;F,H}$ come from diagrams of the form
\be\label{proofd}
\begin{tikzpicture}[scale=1,baseline=13,x=1pt,y=1pt]
\node at (39,23) {\tiny  $d$};
\node at (39,8) {\tiny  $b$};
\node at (-3,30) {\tiny  $c$};
\node at (-3,0) {\tiny  $a$};
\path [draw,thick,postaction={on each segment={end  arrow}}]
 (0,-10) -- (0,-3)  node[near start,left] {\tiny  $E$}
(35,-10) -- (35,0)  node[near start,right] {\tiny  $F$};
\path [draw,thick,postaction={on each segment={mid  arrow}}]
(0,-3) -- (0,33)
 (35,22) -- (0,30)
 (0,33) -- (0,40) node[near end,left] {\tiny  $F'$}
 (35,0) -- (35,30)
 (0,0) -- (35,8)
 (35,30) -- (35,40) node[near end,right] {\tiny  $E'$};
\end{tikzpicture},
\ee
with  internal lines labelled with various subsets of $E\cup F$.

 According to the Venn diagram \eqref{Venn}, we can represent the sets entering~\eqref{proofd} as
\be
E=C\sqcup G\sqcup \dot E,\quad F=C\sqcup H \sqcup \dot F,\quad
E'=C\sqcup H\sqcup \dot E,\quad F'=C\sqcup G \sqcup \dot F.
\ee
We also label the nodes of the diagram \eqref{proofd} by the letters $a,b,c,d$ and consider how different
elements of the sets $E$ and $F$ can move up along the paths of the diagram to make a nontrivial contribution to $\beta_{E,G;F,H}$.

Any element of $\dot E$ must end up in $E'$. This means that the  whole subset $\dot E$
can only move along paths $ab$ and $bd$.
A set $A\subset E\setminus \dot E=G\sqcup C$ of other elements of $E$ that move along  path $ab$ cannot contain elements of $C$. Indeed, an element $x\in C$ moving along $ab$ would meet
 the same element $x\in F$ and the resulting contribution would vanish. Thus, $A$ can be any subset of $G$.

With these observations, we write
\be\label{proofd1}
(-1)^{|E||F|}\beta_{E,G;F,H}=\sum_{A\subseteq G}\gamma_{|G\setminus A|+|C|}\
\begin{tikzpicture}[scale=1,baseline=13,x=1pt,y=1pt]
\path [draw,thick,postaction={on each segment={mid  arrow}}]
(0,-3) -- (0,33) node[midway,left] {\tiny $(G\setminus A)\sqcup C$}
 (45,22) -- (0,30) node[midway,above] {\tiny $A\sqcup \dot F$}
 (0,33) -- (0,40) node[near end,left] {\tiny $ F'$}
 (45,0) -- (45,30) node[midway,right]
{\tiny $\dot E \sqcup A\sqcup F$}
(0,0) -- (45,8) node[midway,above] {\tiny $\dot E \sqcup A$}
 (45,30) -- (45,40) node[near end,right] {\tiny $ E'$};
\path [draw,thick,postaction={on each segment={ end arrow}}]
(0,-10) -- (0,-3)  node[near start,left] {\tiny  $ E$}
(45,-10) -- (45,0)  node[near start,right] {\tiny $ F$};
\end{tikzpicture}\ .
\ee
Using expressions for the product \eqref{eq:prod-comb} and for the coproduct
 \eqref{eq:comult-comb}, we can rewrite~\eqref{proofd1} as follows
\begin{align}
 \sum_{A\subseteq G}\gamma_{|G\setminus A|+|C|}&(-1)^{\theta_{(G\setminus A)\sqcup C,A\sqcup \dot F}+
\theta_{\dot E \sqcup A,F}}
(-p)^{\theta_{(G\setminus A)\sqcup C,\dot E\sqcup A}+
\theta_{A\sqcup \dot F,E'}}\\
=&(-1)^{\theta_{G\sqcup C,\dot E\sqcup \dot F}+\theta_{\dot E,F}+\theta_{\dot F,E'}}\sum_{A\subseteq G}\gamma_{|G\setminus A|+|C|}
p^{\theta_{(G\setminus A)\sqcup C, \dot E\sqcup A}+\theta_{A\sqcup \dot F, E'}}\nonumber\\
=&(-1)^{\theta_{G\sqcup C,\dot E\sqcup \dot F}+\theta_{\dot E,F}+\theta_{\dot F,E'}+|G|+|C|}\sum_{A\subseteq G}(-1)^{|A|}
p^{\theta_{(G\setminus A)\sqcup C,E}+\theta_{A\sqcup \dot F,E'}}\nonumber\\
=&(-1)^{\theta_{G\sqcup C,\dot E\sqcup \dot F}+\theta_{\dot E,F}+\theta_{\dot F,E'}+|G|+|C|}p^{\theta_{G\sqcup C,E}+\theta_{\dot F, E'}}\sum_{A\subseteq G}(-1)^{|A|}
p^{\theta_{A,H}-\theta_{A ,G}}\nonumber\\
=&(-1)^{\theta_{G\sqcup C,\dot E\sqcup \dot F}+\theta_{\dot E,F}+\theta_{\dot F,E'}+|C|}p^{\theta_{G\sqcup C,E}+\theta_{\dot F, E'}}\prod_{A\in \binom{G}{1}}(
p^{\theta_{A,H}-\theta_{A ,G}}-1),\nonumber
\end{align}
where, in the last equality, we use Lemma~\ref{lem:sum-to-product} for converting the sum over $A$ into a product,
so that
\begin{align}
 \beta_{E,G;F,H}=&(-1)^{\theta_{E,F}+\theta_{F,E}+\theta_{G\sqcup C,\dot E\sqcup \dot F}+\theta_{\dot E,F}+\theta_{\dot F,E'}}p^{\theta_{G\sqcup C,E}+\theta_{\dot F, E'}}\alpha_{G,H}\\
 =&(-1)^{\theta_{F,E}+\theta_{F',E'}}p^{\theta_{G\sqcup C,E}+\theta_{\dot F, E'}}\alpha_{G,H}.\nonumber
\end{align}
\end{proof}

\begin{proposition}\label{braiding_el}
For $G\in\binom{E\setminus F}{k}$,  $H\in\binom{F\setminus E}{k}$, we have the following formula for the coefficient in~\eqref{eq:s-coeff}:
\begin{equation}\label{eq:sEGFH}
s_{E,G;F,H}=  \beta_{E,G;F,H},
\end{equation}
which does not vanish for generic $p$ if and only if
\begin{equation}\label{eq:cond_GH}
\theta_{A,H}> \theta_{A,G},\quad  \forall A\in \binom{G}{1}.
\ee
\end{proposition}
\begin{proof}
Equality~\eqref{eq:sEGFH} immediately follows from Proposition~\ref{prop:braiding-MOY}, the definition of the operator $B$ in~\eqref{eq:B-operator}, and Lemma~\ref{lem:betaEGFH}. Conditions~\eqref{eq:cond_GH} follow from Lemma~\ref{lem:alpha0}.
\end{proof}
\begin{remark}
As already mentioned, the simplest form of~\eqref{eq:sEGFH} corresponds to the (elastic) scattering channel~\eqref{eq:0-channel}. In this case, we obtain
\begin{equation}
s_{E,F}=p^{\theta_{F,E}}.
\end{equation}
\end{remark}

\begin{remark}
The equations
 \be
 \tau_k f_{E,F}=0, \quad \forall k>0,
 \ee
 hold if and only if $\theta_{E\setminus F,F\setminus E}=0$.
\end{remark}

\section{Generalised Yetter--Drinfel'd module structures and \texorpdfstring{$R$-matrices}{Lg}} \label{sec:3}
Following the general approach of~\cite{GaroufalidisKashaev2023} to the construction of $R$-matrices from braided Hopf algebras with automorphisms, we observe that for any nonzero scalar $t\in \lcf_{\ne0}$, there exists a diagonal automorphism $\phi_t$ of $\Lambda_p(V)$, which acts on elements of $V$ by multiplication by $t$. This means that it acts on the set-theoretic basis according to the formula
 \begin{equation}\label{eq:diag-aut}
 \phi_tf_E=t^{|E|}f_E,\quad \forall E\in \fP(\B),
\end{equation}
with the following diagrammatic representation:
\be\phi_t=\
\begin{tikzpicture}[baseline=12]
 \node (f) at (0,1/2) [autom]{} ;
 \path[draw,thick,postaction={on each segment={mid arrow}}]
 (0,0)--(f)--(0,1);
\end{tikzpicture}\ .
\ee
In the particular case with $t=p$, this automorphism has already been used in Section~\ref{sec:5}, see definition~\eqref{eq:diag-aut-t=p}.

Associated with this automorphism, there exists a right generalised Yetter--Drinfel'd module structure on $\Lambda_p(V)$ given by the coproduct as the right coaction, and the right action
\begin{equation}\label{eq:r-action}
\lambda_R=\nabla^{(2)}((S\phi_t)\otimes\operatorname{id}_{\Lambda_p(V)\otimes \Lambda_p(V)})(\hat\tau\otimes\operatorname{id}_{\Lambda_p(V)})(\operatorname{id}_{\Lambda_p(V)}\otimes \Delta)\colon  \Lambda_p(V)\otimes \Lambda_p(V)\to\Lambda_p(V),
\end{equation}
with the graphical description
\be\label{faction}
\lambda_R=
\begin{tikzpicture}[xscale=2,yscale=1.7,baseline=12,x=15pt,y=15pt]
\node (d) at (0,-1+1.6)[coact]{};
\coordinate (t1) at (-.5,-2+1.7);
\coordinate (t2) at (.5,-2+1.7);
\coordinate (b) at (0,0+1.7);
\path [draw,thick,postaction={on each segment={mid arrow}}]
(d) -- node [near end,right] {
} (b) ;
\path [draw,thick,postaction={on each segment={mid  arrow}}]
(t1)  to  [in=-135,out=90] node [near start,left] {
} (d)
(t2) to [in=-45,out=90] node [near start,right] {
} (d);
\end{tikzpicture}
=
\begin{tikzpicture}[xscale=.6,yscale=.6,baseline=8,x=4.5pt,y=4.5pt]
\coordinate (t1) at (5,-5);
\coordinate (t2) at (-5,-5);
\coordinate (t3) at (5,7);
\node (s) at (-5,4)[antipode]{} ;
\node (f) at (-5,7)[autom]{};
\coordinate (pr) at (5,-2);
\coordinate (cp) at (0,12);
\coordinate (b) at (0,15);
\path [draw,thick,postaction={on each segment={mid arrow}}]
(pr) to  [in=-45,out=45]  (t3) (pr) to  [in=-90,out=135] (s);
\draw[line width=3pt, white] (t2) to  [out=90,in=-135]  (4.5,6.5);
\path [draw,thick,postaction={on each segment={mid  arrow}}]
(f) to[in=-150,out=90] (cp) (t3) to [in=-30,out=90] (cp)
(t1) -- (pr) (cp) --(b) (t2) to [out=90, in=-135] (t3);
\draw[thick] (f)--(s);
\end{tikzpicture},
\ee
and the corresponding $R$-matrix
\begin{equation}\label{eq:r-mat}
\rho=(\phi_t\otimes\lambda_R)(\hat\tau\otimes\operatorname{id})
(\operatorname{id}\otimes\Delta)\colon\Lambda_p(V)\otimes\Lambda_p(V)\to\Lambda_p(V)\otimes\Lambda_p(V),
\end{equation}
with the graphical description
\begin{equation}\label{eq:r-mat-graph}
\rho=\begin{tikzpicture}[baseline=22]
\coordinate (t1) at (0,2);
\coordinate (t2) at (1,2);
\coordinate (b1) at (0,0);
\coordinate (b2) at (1,0);
\node (f) at (0, 3/2)[autom]{};
\coordinate (cac) at (1,1/2);
\node (ac) at (1,3/2)[coact]{};
\draw[thick, start arrow] (cac) to [out=135,in=-90] (f);
\draw[thick, mid arrow] (f)--(t1);
\draw[thick, mid arrow] (ac)--(t2);
\draw[thick, mid arrow] (b2)--(cac);
\draw[line width=3pt,white] (b1) to [out=90,in=-135] (ac);
\draw[thick, start arrow] (b1) to [out=90,in=-135] (ac);
\draw[thick, mid arrow] (cac) to [out=45,in=-45] (ac);
\end{tikzpicture}.
\end{equation}

In what follows, we use the notation
$x\triangleleft y:=\lambda_R(x\otimes y)$.
\begin{proposition}
The right action~\eqref{eq:r-action}, \eqref{faction} of $\Lambda_p(V)$ on itself is given by
 \begin{equation}\label{eq:right-action-general}
 f_E\triangleleft f_F=f_Ef_F (tp^{|E|};p)_{|F|},\quad \forall E,F\in\fP(\B).
\end{equation}
\end{proposition}
\begin{proof}
Composing with $\pi_m\otimes\pi_n$, we have
\begin{equation}
\lambda_R(\pi_m\otimes\pi_n)=
\begin{tikzpicture}[scale=0.7,baseline=12,x=4.5pt,y=4.5pt]
\coordinate (t1) at (5,-5);
\coordinate (t2) at (-5,-5);
\coordinate (t3) at (5,7);
\node (s) at (-5,3.5)[antipode]{} ;
\node (f) at (-5,7)[autom]{};
\coordinate (pr) at (5,-2);
\coordinate (cp) at (0,12);
\coordinate (b) at (0,15);
\path [draw,thick,postaction={on each segment={mid arrow}}]
(pr) to  [in=-45,out=45]  (t3) (pr) to  [in=-90,out=135] (s);
\draw[line width=3pt, white] (t2) to  [out=90,in=-135]  (4.5,6.5);
\path [draw,thick,postaction={on each segment={mid  arrow}}]
(f) to[in=-150,out=90] (cp) (t3) to [in=-30,out=90] (cp)
(t1) -- (pr) (cp) --(b) (t2) to [out=90, in=-135] (t3);
\draw[thick] (f)--(s);
\node at (-6.5,-4)  {\tiny$m$};
\node at (6.5,-4)  {\tiny$n$};
\node at (4,14) {
};
\end{tikzpicture}\
= \sum_{k=0}^n t^k\gauss{k}\
\begin{tikzpicture}[scale=0.7,baseline=12,x=4.5pt,y=4.5pt]
\coordinate (t1) at (5,-5);
\coordinate (t2) at (-5,-5);
\coordinate (t3) at (5,7);
\coordinate (s)  at (-5,3.5);
\coordinate (f) at (-5,7);
\coordinate (pr) at (5,-2);
\coordinate (cp) at (0,12);
\coordinate (b) at (0,15);
\path [draw,thick,postaction={on each segment={mid arrow}}]
(pr) to  [in=-45,out=45]  (t3) (pr) to  [in=-90,out=135] (s);
\draw[line width=3pt, white] (t2) to  [out=90,in=-135]  (4.5,6.5);
\path [draw,thick,postaction={on each segment={mid  arrow}}]
(f) to[in=-150,out=90] (cp) (t3) to [in=-30,out=90] (cp)
(t1) -- (pr) (cp) --(b) (t2) to [out=90, in=-135] (t3);
\draw[thick] (f)--(s)node [midway,left] {\tiny $k$};
\node at (-6.5,-4)  {\tiny $m$};
\node at (6.5,-4)  {\tiny$n$};
\node at (4,14) {
};
\end{tikzpicture}\ ,
\end{equation}
where, in the second equality, we use the fact that the braiding $\hat\tau$ preserves the degree along the strands.
Using the identity
\be\label{braiding-tr}
\begin{tikzpicture}[scale=1.5,baseline=20,x=1pt,y=1pt]
\path [draw,thick,postaction={on each segment={mid end arrow}}]
(20,5) --(0,25) ;
\draw[line width=3pt, white] (5,10) -- (15,20);
\path [draw,thick,postaction={on each segment={mid end arrow}}]
(0,5) -- (20,25)  ;
\path [draw,thick,postaction={on each segment={mid  arrow}}]
(20,5)--(20,25) to[out=45,in=-90] (23,30) (23,0) to[in=-45,out=90] (20,5);
\draw[thick] (-3,30) to[out=-90,in=-45] (0,25) (-3,0) to[out=90,in=-135](0,5);
\end{tikzpicture}
\ =\
\begin{tikzpicture}[scale=1.5,baseline=20,x=1pt,y=1pt]
\node at (0,15) [antipode]{};
\path [draw,thick,postaction={on each segment={mid  arrow}}]
(0,25) to[in=-90,out=-45] (-3,30) (-3,0) to[out=90,in=-135] (0,5)
(0,16.3)--(0,25) (20,5)--(10,10)--(10,20)--(0,25) (10,20)--(20,25) (0,5)--(10,10) ;
\draw[thick]  (20,5) to [out=-45,in=90] (23,0)
(0,5) -- (0,13.7)   (20,25) to [out=45,in=-90] (23,30);
\end{tikzpicture},
\ee
we also have
\be
\begin{tikzpicture}[scale=0.7,baseline=12,x=4.5pt,y=4.5pt]
\coordinate (t1) at (5,-5);
\coordinate (t2) at (-5,-5);
\coordinate (t3) at (5,7);
\coordinate (s)  at (-5,3.5);
\coordinate (f) at (-5,7);
\coordinate (pr) at (5,-2);
\coordinate (cp) at (0,12);
\coordinate (b) at (0,15);
\path [draw,thick,postaction={on each segment={mid arrow}}]
(pr) to  [in=-45,out=45]  (t3) (pr) to  [in=-90,out=135] (s);
\draw[line width=3pt, white] (t2) to  [out=90,in=-135]  (4.5,6.5);
\path [draw,thick,postaction={on each segment={mid  arrow}}]
(f) to[in=-150,out=90] (cp) (t3) to [in=-30,out=90] (cp)
(t1) -- (pr) (cp) --(b) (t2) to [out=90, in=-135] (t3);
\draw[thick] (f)--(s)node [midway,left] {\tiny $k$};
\node at (-6.5,-4)  {\tiny $m$};
\node at (6.5,-4)  {\tiny$n$};
\node at (4,14) {
};
\end{tikzpicture}
\ = \
\begin{tikzpicture}[scale=0.7,baseline=12,x=4.5pt,y=4.5pt]
\coordinate (t1) at (5,-5);
\coordinate (t2) at (-5,-5);
\coordinate (t3) at (5,6);
\coordinate (t4) at (-4,3.5);
\coordinate (t5) at (-3.5,-2.7);
\coordinate (t6) at (0,2);
\coordinate (t7) at (0,-1);
\node (s1) at (-5,0.1) [antipode]{};
\coordinate (s) at (-5,5);
\coordinate (f) at (-5,7);
\coordinate (pr) at (5,-2);
\coordinate (cp) at (0,12);
\coordinate (b) at (0,15);
\path [draw,thick,postaction={on each segment={mid  arrow}}]
(t7) -- (t6) (f) to[in=-150,out=90] (cp) (t3) to  [in=-30,out=90]  (cp)
(t1) to [out=90,in=-20] (t7) (cp) -- (b) (t5) to [out=45,in=-160] (t7)
 (s1) to [out=90,in=-135] (t4)
 (t6) to  [in=-90,out=150] (s)
(t2) to [out=90, in=-135] (t5);
\draw[thick]   (s1) to [out=-90,in=135] (t5);
\draw[thick]  (t3) to[out=-90,in=30] (t6) (f)--(s)node [midway,left] {\tiny $k$};
\node at (-6.5,-4)  {\tiny $m$};
\node at (6.5,-4)  {\tiny $n$};
\end{tikzpicture}
\ =
\sum_{i=0}^{\min(k,m)}\gamma_i\
\begin{tikzpicture}[scale=0.7,baseline=12,x=4.5pt,y=4.5pt]
\coordinate (t1) at (5,-5);
\coordinate (t2) at (-5,-5);
\coordinate (t3) at (5,6);
\coordinate (t4) at (-4,3.5);
\coordinate (t5) at (-3.5,-2.7);
\coordinate (t6) at (0,2);
\coordinate (t7) at (0,-1);
\coordinate (s1) at (-5,0.1);
\coordinate (s) at (-5,5);
\coordinate (f) at (-5,7);
\coordinate (cp) at (0,12);
\coordinate (b) at (0,15);
\path [draw,thick,postaction={on each segment={mid  arrow}}]
(t6) to  [in=-90,out=150] (s) (s1) to [out=90,in=-135] (t4)
(cp) --(b) (t3) to  [in=-30,out=90]  (cp) (f) to[in=-150,out=90] (cp)
(t1) to [out=90,in=-20] (t7) (t5) to [out=45,in=-160] (t7)
(t7)--(t6) node [midway,right] {
};
\draw[thick] (s1) to [out=-90,in=135] (t5) (t2) to [out=90, in=-135] (t5);
\node at (-6,0) {\tiny $i$};
\draw[thick] (t3) to[out=-90,in=30] (t6) (t1) to [out=90,in=-20] (t7)
(f)--(s) node [midway,left] {\tiny $k$};
\node at (-6.5,-4)  {\tiny $m$};
\node at (6.5,-4)  {\tiny $n$};
\end{tikzpicture}.
\ee

\vspace{0.3cm}
Using associativity and \eqref{eq:bubleMOY}, we can transform the last diagram as follows:
\begin{equation}
\begin{tikzpicture}[scale=0.7,baseline=12,x=4.5pt,y=4.5pt]
\coordinate (t1) at (5,-5);
\coordinate (t2) at (-5,-5);
\coordinate (t3) at (5,6);
\coordinate (t4) at (-4,3.5);
\coordinate (t5) at (-3.5,-2.7);
\coordinate (t6) at (0,2);
\coordinate (t7) at (0,-1);
\coordinate (s1) at (-5,0.1);
\coordinate (s) at (-5,5);
\coordinate (f) at (-5,7);
\coordinate (cp) at (0,12);
\coordinate (b) at (0,15);
\path [draw,thick,postaction={on each segment={mid  arrow}}]
(t6) to  [in=-90,out=150] (s) (s1) to [out=90,in=-135] (t4) ;
\path [draw,thick,postaction={on each segment={mid arrow}}]
(cp) --(b) (t3) to  [in=-30,out=90]  (cp) (f) to[in=-150,out=90] (cp)
(t1) to [out=90,in=-20] (t7) (t5) to [out=45,in=-160] (t7)
(t7)--(t6) node [midway,right] {
};
\draw[thick] (s1) to [out=-90,in=135] (t5) ;
\node at (-6,0) {\tiny $i$};
\draw[thick] (t3) to[out=-90,in=30] (t6);
\path [draw,thick,postaction={on each segment={mid  arrow}}]
 (t2) to [out=90, in=-135] (t5) ;
\draw[thick] (t1) to [out=90,in=-20] (t7);
\draw[thick] (f)--(s) node [midway,left] {\tiny $k$};
\node at (-6.5,-4)  {\tiny $m$};
\node at (6.5,-4)  {\tiny $n$};
\end{tikzpicture}
=
\begin{tikzpicture}[scale=0.7,baseline=12,x=4.5pt,y=4.5pt]
\coordinate (t1) at (5,-5);
\coordinate (t2) at (-5,-5);
\coordinate (t4) at (0,7);
   \coordinate (t5) at (-3.5,-2.7);
\coordinate (t6) at (0,2);
\coordinate (t7) at (0,-1);
\coordinate (cp) at (0,12);
\coordinate (b) at (0,15);
\path [draw,thick,postaction={on each segment={mid  arrow}}]
 (t5) to [out=110,in=-135] node[midway,left] {\tiny $i$} (cp)
 (t4)-- (cp) (-1.32,5) -- (-1.32,5.2) (1.32,5) -- (1.32,5.2)
 (t1) to [out=90,in=-20] (t7);
\path [draw,thick,postaction={on each segment={mid  arrow}}]
(t7) --  node [midway,right] {\tiny $m+n-i$}(t6) (cp) --(b)
 (t5) to [out=45,in=-160] (t7);
\path [draw,thick,postaction={on each segment={mid arrow}}]
 (t2) to [out=90, in=-135] (t5) ;
 \draw[thick] (t6) to[out=160,in=-160] (t4)
  (t6) to [out=20,in=-20] node [midway,right] {\tiny $m+n-k$
  } (t4);
\node at (-6.5,-4)  {\tiny $m$};
\node at (6.5,-4)  {\tiny $n$};
\end{tikzpicture}
=
\qbinom{n+m-i}{k-i}{p}
\begin{tikzpicture}[scale=0.7,baseline=12,x=4.5pt,y=4.5pt]
\coordinate (t1) at (5,-5);
\coordinate (t2) at (-5,-5);
\coordinate (t5) at (-3.5,-2.7);
\coordinate (t7) at (0,-1);
\coordinate (cp) at (0,12);
\coordinate (b) at (0,15);
\path [draw,thick,postaction={on each segment={mid  arrow}}]
(t5) to [out=110,in=-135] node[midway,left] {\tiny $i$} (cp)
(t7) --  node [midway,right] {
}(cp) ;
\path [draw,thick,postaction={on each segment={mid  arrow}}]
(cp) --(b) (t5)  to [out=45,in=-160] (t7) (t1) to [out=90,in=-20] (t7);
\path [draw,thick,postaction={on each segment={mid arrow}}]
 (t2) to [out=90, in=-135] (t5);
\node at (-6.5,-4)  {\tiny $m$};
\node at (6.5,-4)  {\tiny $n$};
\end{tikzpicture}
\end{equation}
and
\begin{equation}
\begin{tikzpicture}[scale=0.7,baseline=12,x=4.5pt,y=4.5pt]
\coordinate (t1) at (5,-5);
\coordinate (t2) at (-5,-5);
\coordinate (t5) at (-3.5,-2.7);
\coordinate (t7) at (0,-1);
\coordinate (cp) at (0,12);
\coordinate (b) at (0,15);
\path [draw,thick,postaction={on each segment={mid  arrow}}]
(t5) to [out=110,in=-135] node[midway,left] {\tiny $i$} (cp)
(t7) --  node [midway,right] {
}(cp) ;
\path [draw,thick,postaction={on each segment={mid  arrow}}]
(cp) --(b) (t5)  to [out=45,in=-160] (t7) (t1) to [out=90,in=-20] (t7);
\path [draw,thick,postaction={on each segment={mid arrow}}]
 (t2) to [out=90, in=-135] (t5);
\node at (-6.5,-4)  {\tiny $m$};
\node at (6.5,-4)  {\tiny $n$};
\end{tikzpicture}
=
\begin{tikzpicture}[scale=0.7,baseline=12,x=4.5pt,y=4.5pt]
\coordinate (t1) at (5,-5);
\coordinate (t2) at (-5,-5);
\coordinate (t3) at (0,6);
\coordinate (t4) at (-4.5,-1);
\coordinate (t5) at (-2.2,3.5);
\coordinate (b) at (0,15);
\path [draw,thick,postaction={on each segment={mid  arrow}}]
(t4) to[out=130,in=170] node[midway,left] {\tiny $i$} (t5)
(t4) to [out=10, in=-60] (t5)
(t2) to [out=90, in=-100] (t4) (t5) to [out=60, in=-135] (t3)
(t3) -- (b) (t1) to [out=90,in=-45] (t3);
\node at (-6.5,-4)  {\tiny $m$};
\node at (6.5,-4)  {\tiny $n$};
\end{tikzpicture}
=
\qbinom{m}{i}{p}
\begin{tikzpicture}[scale=0.7,baseline=12,x=4.5pt,y=4.5pt]
\coordinate (t1) at (5,-5);
\coordinate (t2) at (-5,-5);
\coordinate (t3) at (0,6);
\coordinate (b) at (0,15);
\path [draw,thick,postaction={on each segment={mid arrow}}]
 (t2) to [out=90, in=-135] (t3) (t3) -- (b)
 (t1) to [out=90,in=-45]  (t3);
\node at (-6.5,-4)  {\tiny $m$};
\node at (6.5,-4)  {\tiny $n$};
\end{tikzpicture}.
\end{equation}
Putting everything together, we come to the expression
 \be
  f_E\triangleleft f_F=f_Ef_F u_{|E|,|F|},
  \ee
  where
 \be
  u_{m,n}:=\sum_{k=0}^{n}t^k\gauss{k}v_{k,m,n},\quad v_{k,m,n}:=\sum_{i=0}^{\min(k,m)}\gauss{i}\qbinom{m}{i}{p}\qbinom{m+n-i}{k-i}{p}.
  \ee
After simplification we can calculate $v_{k,m,n}$ using a $q$-Vandermonde sum (\cite{MR2128719}, Appendix (II.7))
\be
v_{k,m,n}=\qbinom{m+n}{k}{p}\sum_{i=0}^{\min(k,m)}\frac{(p^{-m};p)_i(p^{-k};p)_i}{(p)_i(p^{-m-n};p)_i}
p^{(k-n)i}=\qbinom{m+n}{k}{p}\frac{(p^{-n};p)_k}{(p^{-m-n};p)_k},
\ee
so that
\be
u_{m,n}=\sum_{k=0}^{n}t^k\gauss{k}\qbinom{m+n}{k}{p}\frac{(p^{-n};p)_k}{(p^{-m-n};p)_k}=
\sum_{k=0}^n t^k \gauss{k}\qbinom{n}{k}{p} p^{mk}=(t p^m;p)_n,
\ee
where the last equality is by the $q$-binomial formula~\eqref{eq:qbinom}.
\end{proof}
Let us introduce the following diagrammatic notation
\be\label{rightF}
\begin{tikzpicture}[baseline={(0,-0.1)}
,x=1pt,y=1pt]
 \coordinate (n) at (0,20);
    \coordinate (s) at (0,-20);
\draw[thick] \bell{0,0}{6}{3};
         \draw [thick] (s) --(0,-3);
         \draw[thick,mid arrow] (0,3)--(n) node[midway,right] {
         };
\end{tikzpicture}
\ =
\sum_{n\ge0 }(t;p)_n\quad
\begin{tikzpicture}[baseline={(0,-0.1)}
,x=1pt,y=1pt]
 \coordinate (n) at (0,20);
    \coordinate (s) at (0,-20);
         \draw[thick,mid arrow] (s)--(n) node[midway,right] {\tiny  $n$};
\end{tikzpicture},
\qquad
\begin{tikzpicture}[baseline={(0,-0.1)}
,x=1pt,y=1pt]
 \coordinate (n) at (0,20);
    \coordinate (s) at (0,-20);
\draw[fill=gray,thick] \bell{0,0}{6}{3};
         \draw [thick] (s) --(0,-3);
         \draw[thick,mid arrow] (0,3)--(n) node[midway,right] {
         };
\end{tikzpicture}
\ = \sum_{n\ge0}
\frac{1}{(t;p)_n}\quad
\begin{tikzpicture}[baseline={(0,-0.1)}
,x=1pt,y=1pt]
 \coordinate (n) at (0,20);
    \coordinate (s) at (0,-20);
         \draw[thick,mid arrow] (s)--(n) node[midway,right] {\tiny  $n$};
\end{tikzpicture}.
\ee
We have the obvious diagrammatic equalities
\be
\begin{tikzpicture}[baseline={(0,-0.1)}
,x=1pt,y=1pt]
 \coordinate (n) at (0,20);
    \coordinate (s) at (0,-20);
\draw[thick] \bell{0,5}{6}{3};
\draw[fill=gray,thick] \bell{0,-5}{6}{3};
         \draw [thick] (s) --(0,-8) (0,2)--(0,-2);
         \draw[thick,mid arrow] (0,8)--(n) node[midway,right] {
         };
\end{tikzpicture}
\ = \  \
\begin{tikzpicture}[baseline={(0,-0.1)}
,x=1pt,y=1pt]
 \coordinate (n) at (0,20);
    \coordinate (s) at (0,-20);
         \draw[thick,mid arrow] (s)--(n) node[midway,right] {
         };
\end{tikzpicture}
=\ \begin{tikzpicture}[baseline={(0,-0.1)}
,x=1pt,y=1pt]
 \coordinate (n) at (0,20);
    \coordinate (s) at (0,-20);
\draw[fill=gray,thick] \bell{0,5}{6}{3};
\draw[thick] \bell{0,-5}{6}{3};
         \draw [thick] (s) --(0,-8) (0,2)--(0,-2);
         \draw[thick,mid arrow] (0,8)--(n) node[midway,right] {
         };
\end{tikzpicture}\ .
\ee
\begin{lemma}\label{rightMOY}
The right action $x\triangleleft y$ given by \eqref{eq:right-action-general}
can be written in the following diagrammatic form
\be
\begin{tikzpicture}[scale=2,baseline=20,x=15pt,y=15pt]
\node (d) at (0,-1+1.6)[coact]{};
\coordinate (t1) at (-.5,-2+1.7);
\coordinate (t2) at (.5,-2+1.7);
\coordinate (b) at (0,0+1.7);
\path [draw,thick,postaction={on each segment={mid end arrow}}]
(d) --  (b) ;
\path [draw,thick,postaction={on each segment={mid  arrow}}]
(t1)  to  [in=-135,out=90] node [near start,left] {
} (d)
(t2) to [in=-45,out=90] node [near start,right] {
} (d);
\end{tikzpicture}
=
\begin{tikzpicture}[scale=2,baseline=20,x=15pt,y=15pt]
\coordinate (d) at (0,0.6);
\coordinate (t1) at (-.5,-0.3);
\coordinate (t2) at (.5,-0.3);
\coordinate (b) at (0,0+1.7);
\path [draw,thick,postaction={on each segment={mid arrow}}]
(0,1.1) --  (b) (-0.2,0.44) --(-0.15,0.48);
\path [draw,thick,postaction={on each segment={mid  arrow}}]
(t2) to [in=-45,out=90] node [near start,right] {
} (d);
\draw[thick] (t1)  to  [in=-144,out=90] node [near start,left] {
} (d);

\draw[fill=gray,thick,rotate around={-36:(-0.35,0.2)}] \bell{-0.36,0.2}{0.2}{0.1};
\draw[thick] \bell{0,1}{0.2}{0.1} (d)--(0,0.9);

\end{tikzpicture}.
\ee
\end{lemma}
\begin{proof}
This immediately follows from definition \eqref{rightF}
and the identity
\be\label{eq:poch-add}
(t;p)_{m+n}=(t;p)_n(tp^n;p)_m.
\ee
\end{proof}

\begin{theorem}\label{thm:big-R-matrix}
 The $R$-matrix~\eqref{eq:r-mat}, \eqref{eq:r-mat-graph} acts as follows
\begin{equation}\label{eq:rho-act}
\rho f_{E,F}=\sum_{k=0}^{\min(|E|,|F|)}t^k\gauss{k}W_{|F|-k}\big(tp^{|E|},tL;p\big)R^{\langle |E|-k\rangle _p}f_{E,F},
\end{equation}
where
\be
W_n(x,y;p):=\sum_{i=0}^n(x;p)_{n-i}y^{\langle i\rangle_p}.
\ee
This corresponds to the MOY diagrammatic equation
\begin{equation}
\begin{tikzpicture}[yscale=2,baseline=-3]
\node (x)[hvector] {  $\rho $};
\path [draw,thick,postaction={on each segment={mid arrow}}]
 (x.north west) to [out=135,in=-90]node[midway,left] {\tiny $i$}+(-10pt,10pt)
(x.north east) to [out=45,in=-90]node[midway,right] {\tiny $m+n-i$}+(10pt,10pt)
(x.south west)+(-10pt,-10pt) to [out=90,in=-135]  node[midway,left] {\tiny $m$} (x.south west)
(x.south east)+(10pt,-10pt)  to [out=90,in=-45]  node[midway,right] {\tiny $n$} (x.south east);
\end{tikzpicture}
=
t^i(tp^m;p)_{n-i}\sum_{k=0}^{\min(i,m)}\gauss{k}
\begin{tikzpicture}[yscale=.37,xscale=1,baseline=24]
\coordinate (wn) at (0,5);
\coordinate (en) at (1,5);
\coordinate (ws) at (0,0);
\coordinate (es) at (1,0);
\coordinate (w1) at (0,1);
\coordinate (w2) at (0,2);
\coordinate (w3) at (0,3);
\coordinate (w4) at (0,4);
\coordinate (e1) at (1,1);
\coordinate (e2) at (1,2);
\coordinate (e3) at (1,3);
\coordinate (e4) at (1,4);
\path [draw,thick,postaction={on each segment={mid arrow}}]
 (w1)--(e2) (e3)--(w4) (e2)--(e3)
 (ws)--(w1) node[midway,left] {\tiny $m$}
 (w1)--(w4)  node[midway,left] {\tiny $k$}
(w4)--(wn)node[midway,left] {\tiny $i$}
 (es)--(e2)  node[midway,right] {\tiny $n$}
 (e3)--(en)  node[midway,right] {\tiny $m+n-i$} ;
\end{tikzpicture}.
\end{equation}

\end{theorem}
\begin{proof}
 By using~\eqref{eq:r-mat}, \eqref{eq:right-action-general},
 the $R$-matrix $\rho$ can be represented by the following diagrams:
\be\label{eq:rho-grafical-eqs}
\rho =
 \begin{tikzpicture}[xscale=.6, baseline=-1,x=1pt,y=1pt]
\coordinate(ze) at (0,0);
\coordinate(sw) at (-30,-30);
\coordinate(nw) at (-30,30);
\coordinate(se) at (30,-30);
\coordinate(ne) at (30,30);
\coordinate (bot) at (26,-19);
\node (d) at (26,19)[coact]{};
\node (f) at (-26,19)[autom]{};
\path [draw,thick,postaction={on each segment={mid arrow}}]
(bot) to[out=120,in=-45] (ze)
(se) to[out=100,in=120] node[near start,right] {
} (bot)
(ze) to[out=135,in=-46] (f) to[out=122,in=-90] (nw);
\draw[line width=3pt, white]
(-5,-5) to[out=90,in=-135]   (ze)  (ze) to[out=45,in=-90] (5,5) ;
\path [draw,thick,postaction={on each segment={mid arrow}}]
(ze) to[out=45,in=-134] (d) (d) to[out=58,in=-90] (ne) (bot) -- (26,18);
\path [draw,thick,postaction={on each segment={mid end arrow}}]
(sw) to[out=90,in=-135]  (ze);
 \end{tikzpicture}
\quad =\quad
 \begin{tikzpicture}[baseline=-1,xscale=.83,x=1pt,y=1pt]
\coordinate(ze) at (0,0);
\coordinate(sw) at (-30,-30);
\coordinate(nw) at (-30,30);
\coordinate(se) at (30,-30);
\coordinate(ne) at (30,30);
\coordinate (d) at (24,17);
\coordinate (d1) at (28,21);
\coordinate (d2) at (18,-12);
\coordinate (d3) at (18,12);
\node (f) at (-26,19)[autom]{};
\path [draw,thick,postaction={on each segment={mid arrow}}]
(se)to[out=90,in=138](d2) (ze)to[out=135,in=-46](f) (f)to[out=122,in=-90](nw);
\draw[thick]  (d2) to[out=134,in=-45] (ze);
\draw[line width=3pt, white]
(-5,-5) to[out=90,in=-135] (ze) (ze) to[out=45,in=-90] (5,5) ;
\path [draw,thick,postaction={on each segment={mid arrow}}]
(d3) to[out=45,in=-134] (d) (d1) to[out=58,in=-90] (ne) (d2) --(d3);
\draw[thick]  (ze) to[out=45,in=-134] (d3);
\path [draw,thick,postaction={on each segment={mid end arrow}}]
(sw) to[out=90,in=-135] (ze);
\draw[thick,rotate around={-44:(26,19)}] \bell{26,19}{6}{3};
\draw[fill=gray,thick,rotate around={-44:(-26,-19)}] \bell{-26,-19}{6}{3};
 \end{tikzpicture}
\quad =\quad
 \begin{tikzpicture}[baseline=-1,x=1pt,y=1pt]
\coordinate(ze1) at (0,-5);
\coordinate(ze2) at (0,5);
\coordinate(sw) at (-20,-30);
\coordinate(nw) at (-20,30);
\coordinate(se) at (20,-30);
\coordinate(ne) at (20,30);
\coordinate (d) at (15,15);
\coordinate (d1) at (17.5,20);
\coordinate (d2) at (-12.5,-14);
\coordinate (d3) at (-12.5,14);
\node (s) at (-12.5,0)[antipode]{} ;

\node (f) at (-16,18)[autom]{};
\path [draw,thick,postaction={on each segment={mid arrow}}]
 (ze2)to[out=160,in=-46](f) (ze1)--(ze2) (s)--(d3)
(f)to[out=122,in=-90] node[near end,left] {
} (nw)
(ze2) to[out=20,in=-120] (d);
\path [draw,thick,postaction={on each segment={mid arrow}}]
 (d1) to[out=58,in=-90] node[near end,right] {
 } (ne) ;
\path [draw,thick,postaction={on each segment={mid end arrow}}]
(se) to[out=110,in=-20] (ze1)
(sw) to[out=70,in=-120] (d2) to[out=60,in=-160] (ze1);
\draw[thick] (d2)--(s);
\draw[thick,rotate around={-26:(16.5,17.5)}] \bell{16.5,17.5}{6}{3};
\draw[fill=gray,thick,rotate around={-26:(-17.5,-24)}] \bell{-17.5,-24}{6}{3};
 \end{tikzpicture},
\ee
where, in the second equality,  we use Lemma~\ref{rightMOY} and the fact that the braiding $\hat\tau$ preserves the degree along the strands, and, in the third equality, we use identity~\eqref{braiding-tr}.
Restricting $\rho$ to $\Lambda_p^m(V)\otimes \Lambda_p^n(V)$, using Proposition~\ref{prop:LkRk} and \eqref{eq:poch-add}, we obtain
\be\label{rhoproof}
\rho\vert_{\Lambda_p^m(V)\otimes \Lambda_p^n(V)}= \sum_{i=0}^n t^i(tp^m;p)_{n-i}\sum_{k=0}^{\min(i,m)}\gauss{k}L^{\langle i-k\rangle _p}R^{\langle m-k\rangle _p}\vert_{\Lambda_p^m(V)\otimes \Lambda_p^n(V)},
\ee
and~\eqref{eq:rho-act} follows
after exchanging the order of summation.
\end{proof}
We define the matrix coefficients of $\rho$ as follows:
\begin{equation}\label{eq:BigRmatrix}
\rho f_{E,F}=(-1)^{|E||F|}\sum r_{E,G;F,H} f_{F',E'}
\end{equation}
where the summation runs over $G\subseteq E\setminus F$ and $H\subseteq F\setminus E$ such that $|H|\ge|G|$, and the map
 \be
 \theta_{\{\cdot\},H}\colon G\to \Z_{\ge 0}
 \ee
 is injective with the image contained in $\Z_{> 0}$, and
\be
E':=(E\setminus G)\sqcup H,\quad F':=(F\setminus H)\sqcup G.
\ee
\begin{proposition}
 We have the following formula for the coefficients of $\rho$ defined in~\eqref{eq:BigRmatrix}:
 \begin{equation}\label{eq:bigR-coefficients}
 r_{E,G;F,H} =t^{|F'|}(tp^{|E|};p)_{|H|-|G|}\beta_{E,G;F,H},
\end{equation}
 where $\beta_{E,G;F,H}$ is given by~\eqref{eq:betaEGFH}.
\end{proposition}
\begin{proof}
Equality~\eqref{eq:bigR-coefficients} is an immediate consequence of the last equality in~\eqref{eq:rho-grafical-eqs} and the definition of the operator $B$ in~\eqref{eq:B-operator} and Lemma~\ref{lem:betaEGFH}.
\end{proof}
\begin{conjecture}\label{conj:1}
 Let $\dim(V)=N$. Then the invariant of long knots $J_{\rho}$ associated with the $R$-matrix~\eqref{eq:r-mat} is of the form
 \be
 J_{\rho}=\operatorname{LG}^{(N)}\!(p,t)\operatorname{id}_{\Lambda_p(V)}
 \ee
 where $\operatorname{LG}^{(N)}\!(p,t)$ is the Links--Gould invariant~\cite{MR1199742,MR1223526} associated with a $2^N$-dimensional representation of the super quantum group
 \( U_q(\mathfrak{gl}(N|1)) \).
\end{conjecture}
The calculations for a few examples of knots for the values $N=2$ and $N=3$ and a comparison with the results of~\cite{MR1822139} are consistent with Conjecture~\ref{conj:1}.

Below, we work out  in detail  the  explicit form of the $R$-matrix defined by~\eqref{eq:r-mat} and determined by~\eqref{eq:BigRmatrix} and~\eqref{eq:bigR-coefficients} for dimensions 1, 2, 3 of the input vector space $V$.

\subsection{The case \texorpdfstring{$\operatorname{dim}(V)=1$.}{Lg}}
For illustrative purposes, we begin by  considering a one-dimensional input vector space $V$, where no nontrivial deformed braided Hopf algebra structures exist.

We have two basis elements $f_0:=f_{\{\}}$ and $f_1:=f_{\{1\}}$ and
the $R$-matrix~\eqref{eq:r-mat} explicitly acts as follows:
\be
(\rho f_{i,j})_{i,j=0}^1=\begin{pmatrix}
 f_{0,0}&(t)_1f_{0,1}+tf_{1,0}\\
f_{0,1}&-tf_{1,1}
\end{pmatrix}.
\ee
The associated (endomorphism-valued) knot invariant, see~\cite{GaroufalidisKashaev2023} and also \cite[Chapter 6.3]{MR4592646}, takes the form
\be
J_\rho=\Delta(t)\operatorname{id}_{\Lambda_p(V)},
\ee where $\Delta(t)$ is the Alexander polynomial with unique normalisation such that
\be
\Delta(t)=\Delta(1/t),\quad \Delta(1)=1.
\ee
\subsection{The case \texorpdfstring{$\operatorname{dim}(V)=2$.}{Lg}}

Denoting the basis elements as
\be
f_0:=f_{\{\}},\quad f_1:=f_{\{1\}},\quad f_2:=f_{\{2\}},\quad f_3:=f_{\{1,2\}},
\ee
we have the degrees
\be
\operatorname{deg}(f_i)=[i],\quad ([i])_{i=0}^3=(0,1,1,2).
\ee

The $R$-matrix $\rho$ splits into three parts: `scattering', `reflection', and `annihilation'
\be
\rho=\rho_s+\rho_{r}+\rho_{a}
\ee
where the scattering part $\rho_s$ acts as
\be
\rho_sf_{i,j}=(-1)^{[i][j]}t^{[j]}p^{m_{j,i}} f_{j,i}
\ee
with the matrix of exponents of $p$
\begin{equation}\label{eq:exponents2}
( m_{i,j})_{i,j=0}^3=
\begin{pmatrix}
0&0&0&0\\
0&0&0&0\\
0&1&0&1\\
0&1&0&1
\end{pmatrix}.
\end{equation}
The reflection part $\rho_r$  acts as
\be
\rho_rf_{i,j}=(-t)^{[i]}r_{i,j} f_{i,j}
\ee
with the matrix of reflection coefficients
\be
(r_{i,j})_{i,j=0}^3=
\begin{pmatrix}
0&(t)_1&(t)_1&(t;p)_2\\
0&0&0&(tp)_1\\
0&(p)_1&0&(tp)_1\\
0&0&0&0
\end{pmatrix},
\ee
and the annihilation part $\rho_a$ is of the form
\be
(\rho_{a} f_{i,j})_{i,j=0}^3\\
=\begin{pmatrix}
0&0&0&t(t)_1(f_{1,2}-pf_{2,1})\\
0&0&(tp)_1f_{0,3}&0\\
0&-(tp)_1f_{0,3}&0&0\\
0&0&0&0
\end{pmatrix}.
\ee
The associated endomorphism-valued knot invariant is expected to take the form
\be
J_\rho=\operatorname{LG}^{(2)}\!(p,t)\operatorname{id}_{\Lambda_p(V)},
\ee
where $\operatorname{LG}^{(2)}\!(p,t)$ is the Links--Gould polynomial~\cite{MR1199742} associated with a 4-dimensional representation of the super quantum group $U_q(\mathfrak{gl}(2|1))$.

\subsection{The case \texorpdfstring{$\operatorname{dim}(V)=3$.}{Lg}}

Denoting the basis elements
\be
 f_0:=f_{\{\}},\quad f_i:=f_{\{i\}},\quad i\in\{1,2,3\},\quad
  f_{1+i+j}:=f_{\{i,j\}},\quad 1\le i<j\le3,\quad f_7:=f_{\{1,2,3\}},
\ee
we have the degrees
\be
\operatorname{deg}(f_i)=[i],\quad ([i])_{i=0}^7=(0,1,1,1,2,2,2,3).
\ee

The $R$-matrix $\rho$  splits into `scattering', `reflection' and `annihilation' parts
\be
\rho=\rho_s+\rho_{r}+\rho_{a}
\ee
where the scattering part $\rho_s$ acts as
\be
\rho_sf_{i,j}=(-1)^{[i][j]}t^{[j]}p^{m_{j,i}} f_{j,i}
\ee
with the matrix of exponents of $p$
\begin{equation}\label{eq:exponents3}
(m_{i,j})_{i,j=0}^7=
\begin{pmatrix}
0&0&0&0&0&0&0&0\\
0&0&0&0&0&0&0&0\\
0&1&0&0&1&1&0&1\\
0&1&1&0&2&1&1&2\\
0&1&0&0&1&1&0&1\\
0&1&1&0&2&1&1&2\\
0&2&1&0&3&2&1&3\\
0&2&1&0&3&2&1&3
\end{pmatrix}.
\end{equation}
The reflection part acts as
\be
\rho_{r} f_{i,j}=(-t)^{[i]}p^{\delta_{[i],2}}r_{i,j}f_{i,j}
\ee
with the matrix of reflection coefficients
\be
(r_{i,j})_{i,j=0}^7
=\begin{pmatrix}
0&(t)_1&(t)_1&(t)_1&(t;p)_2&(t;p)_2&(t;p)_2&(t;p)_3\\
0&0&0&0&(tp)_1&(tp)_1&0&(tp;p)_2\\
0&(p)_1&0&0&(tp)_1&(p;t)_2&(tp)_1&(tp;p)_2\\
0&(p)_1&(p)_1&0&(p^2)_1(tp)_1&(tp)_1&(tp)_1&(tp;p)_2\\
0&0&0&0&0&0&0&(tp^2)_1\\
0&0&0&0&(p)_1&0&0&(tp^2)_1\\
0&0&0&0&(p)_1&(p)_1&0&(tp^2)_1\\
0&0&0&0&0&0&0&0
\end{pmatrix}.
\ee

The annihilation part $\rho_a$, in turn, splits into `decay', `fusion'  and `exchange' parts
\be
\rho_a=\rho_d+\rho_f+\rho_e.
\ee
For the decay part $\rho_d$, we have the following four nontrivial cases:
\be
\rho_d f_{0,1+i+j}=t(t)_1(f_{i,j}-pf_{j,i}),\quad 1\le i<j\le 3,
\ee
and
\be
\rho_d f_{0,7}=t(t;p)_2(f_{1,6}-pf_{2,5}+p^2f_{3,4})+t^2(t)_1(f_{4,3}-pf_{5,2}+p^2f_{6,1}).
\ee
The fusion part $\rho_f$, which is, in a sense, dual to the decay part, has the following  nontrivial cases:
\be
\rho_f f_{i,j}=-\rho_f f_{j,i}=(tp)_1f_{0,1+i+j},\quad 1\le i<j\le 3,
\ee
and
\be
\rho_f f_{i,7-i}=
\begin{cases}
 (-1)^{i-1}(tp;p)_2 f_{0,7}& \text{if } 1\le i\le 3;\\
(-1)^i (tp^2)_1f_{0,7}& \text{if } 4\le i\le 6.
\end{cases}
\ee
The exchange part $\rho_{e}$ has the following nontrivial cases:
\be
\rho_{e} f_{1,6}=tp(tp)_1(pf_{3,4}-f_{2,5}),\quad \rho_{e} f_{1,7}=t^2p(tp)_1(f_{4,5}-pf_{5,4}),
\ee
\be
\rho_{e} f_{2,5}=t^2p(p)_1f_{6,1}-t(tp)_1(f_{1,6}+p^2f_{3,4}),\quad \rho_{e} f_{2,7}=t^2(tp)_1(f_{4,6}-p^2f_{6,4}),
\ee
\be
\rho_{e} f_{3,4}=t(tp)_1 (f_{1,6}-pf_{2,5})+t^2(p)_1(f_{5,2}-pf_{6,1}),\quad \rho_{e} f_{3,7}=t^2(tp)_1 (f_{5,6}-pf_{6,5}),
\ee
\be
\rho_{e} f_{4,5}=t(tp^2)_1f_{1,7},\quad \rho_{e} f_{4,6}=tp(tp^2)_1f_{2,7},
\ee
\be
\rho_{e} f_{5,2}=tp(p)_1f_{3,4},\quad \rho_{e} f_{5,4}=t(tp^2)_1f_{1,7},\quad \rho_{e} f_{5,6}=tp(tp^2)_1f_{3,7},
\ee
\be
\rho_{e} f_{6,1}=t(p)_1(f_{2,5}-pf_{3,4}),\quad \rho_{e} f_{6,4}=-t(tp^2)_1f_{2,7},\quad \rho_{e} f_{6,5}=-tp(tp^2)_1f_{3,7}.
\ee

The associated endomorphism-valued knot invariant is expected to take the form
\be
J_\rho=\operatorname{LG}^{(3)}\!(p,t)\operatorname{id}_{\Lambda_p(V)},
\ee
where $\operatorname{LG}^{(3)}\!(p,t)$ is the Links--Gould polynomial~\cite{MR1223526} associated with an 8-dimensional representation of the super quantum group $U_q(\mathfrak{gl}(3|1))$.

\section{Conclusions and Discussion}

In this work, we have shown that the exterior algebra \( \bigwedge\!\!V \) of a \( \lcf \)-vector space \( V \) of dimension \( N \ge 2 \) admits a one-parameter family of braided Hopf algebra structures \( \Lambda_p(V) \), with \( p \in \lcf^\times \), arising from its realisation as a Nichols algebra associated with a specific Hecke-type \( R \)-matrix (braiding) \( \tau \) on \( V \). The case \( p = 1 \) corresponds to the standard Nichols algebra structure on \( \bigwedge\!\!V \) associated with the ``super'' braiding \eqref{superbr}. 
We provide explicit structure constants for \( \Lambda_p(V) \) with respect to a natural set-theoretic basis. While the algebraic structure of \( \Lambda_p(V) \) remains undeformed, the braiding, coproduct, and antipode vary with \( p \).

Using MOY diagrammatic calculus, we derive a simplified formula for the induced braiding \( \hat{\tau} \) on \( \Lambda_p(V) \), and provide an explicit expression for its matrix coefficients with respect to the set-theoretic basis. Furthermore, we introduce a family of diagonal automorphisms and use them to construct explicit constant solutions to the quantum Yang--Baxter equation over \( \Lambda_p(V) \). Through computations in low dimensions, we provide evidence supporting the conjecture that these solutions yield two-variable Links--Gould polynomial invariants associated with the \( 2^N \)-dimensional representations of the quantum supergroup \( U_q(\mathfrak{gl}(N|1)) \).

Future directions include:
\begin{itemize}
    \item extend our computation of the \( R \)-matrix to derive the associated knot invariants in specific examples, without fixing the dimension \( N \), and study their behaviour as functions of \( N \);
    \item motivated by the \( V_n \) invariants of~\cite{GaroufalidisKashaev2023}, exploring deeper ties with the representation theory of \( U_q(\mathfrak{gl}(N|1)) \) and constructing coloured versions of the Links--Gould invariants.
\end{itemize}

These developments promise to deepen the algebraic and topological applications of Nichols algebras.

\bibliographystyle{plain}

\def\cprime{$'$} \def\cprime{$'$}

\end{document}